\documentclass[a4paper,11pt]{article}
\usepackage[T1]{fontenc}
\usepackage{lmodern}
\usepackage[margin=27mm,top=25mm,bottom=26mm]{geometry}
\usepackage{amsmath,amssymb,amsthm,mathtools}
\usepackage{booktabs,longtable,array,enumitem,microtype}
\usepackage[hidelinks]{hyperref}
\usepackage{xurl,fancyvrb,needspace}
\setlist{itemsep=2pt,topsep=3pt}
\allowdisplaybreaks[2]
\numberwithin{equation}{section}
\newtheorem{theorem}{Theorem}[section]
\newtheorem{lemma}[theorem]{Lemma}
\newtheorem{proposition}[theorem]{Proposition}
\newtheorem{corollary}[theorem]{Corollary}
\theoremstyle{definition}

\theoremstyle{remark}

\newcommand{\T}{\mathbb T}
\newcommand{\Z}{\mathbb Z}
\newcommand{\MA}{\mathfrak M}
\newcommand{\m}{\mathfrak m}
\newcommand{\Sg}{\mathfrak S}
\newcommand{\ind}{\mathbf 1}
\newcommand{\meas}{\operatorname{meas}}
\newcommand{\dd}{\,\mathrm d}

\DeclareMathOperator{\lcm}{lcm}
\hypersetup{pdftitle={Two Primes and Five Mixed Fibonacci--Lucas Summands},pdfauthor={Zhao Shen and Chenghe Xu}}
\title{Two Primes and Five Mixed Fibonacci--Lucas Summands}
\author{Zhao Shen\thanks{E-mail: \texttt{sz1021@csu.edu.cn}.}
\and Chenghe Xu\thanks{Corresponding author. E-mail: \texttt{262138051@csu.edu.cn}.}}
\date{\small School of Mathematics and Statistics, Central South University\\
Changsha, China}
\begin{document}
\maketitle
\begin{abstract}
Using the circle method, we prove that every sufficiently large positive integer
is a sum of two primes, $u$ Fibonacci numbers and $5-u$ Lucas numbers for each
$u\in\{1,2,3,4\}$. All recurrence indices are positive, and terms may be repeated.
\end{abstract}
\noindent\textbf{2020 Mathematics Subject Classification.} Primary 11P32; Secondary 11B39, 11L07.\par
\noindent\textbf{Keywords.} Additive representations, Fibonacci numbers, Lucas numbers, circle method, computer-assisted proof.

\section{Introduction}\label{sec:intro}
Let the Fibonacci and Lucas sequences be defined by
\[
 F_0=0,\quad F_1=1,\quad F_{n+2}=F_{n+1}+F_n,
 \qquad
 L_0=2,\quad L_1=1,\quad L_{n+2}=L_{n+1}+L_n
 \quad(n\ge0).
\]
We use only positive indices. Repeated primes and repeated recurrence terms are allowed. Throughout, $\log$ denotes the natural logarithm.

The addition of a few terms from an exponentially growing sequence is a
classical approximation to the binary Goldbach problem. Linnik~\cite{Linnik}
proved that a fixed number of powers of two suffices for all sufficiently large
even integers. Heath-Brown and Puchta~\cite{HBSP} obtained thirteen
powers of two and developed a large-values argument for the associated sparse
exponential sum. The work of Pintz and Ruzsa~\cite{PRI,PR} combines such estimates
with refined major-arc analysis; their second paper obtains eight powers of two.
These results concern powers of two, not Fibonacci or Lucas summands.

For the two recurrence sequences considered here, the closest direct antecedent
is the following result.
\begin{theorem}[Sun {\cite[Theorems 1.1--1.3]{Sun}}]\label{thm:sun}
Every sufficiently large positive integer can be written in each of the forms
\[
 \begin{aligned}
 p_1+p_2+\sum_{i=1}^{6}F_{r_i},\\
 p_1+p_2+\sum_{j=1}^{7}L_{s_j},\\
 p_1+p_2+\sum_{i=1}^{3}F_{r_i}+\sum_{j=1}^{3}L_{s_j},
 \end{aligned}
\]
where $p_1,p_2$ are primes and all recurrence indices are positive integers.
Repeated summands are allowed.
\end{theorem}
Our result has five recurrence summands and covers every composition in which
both sequences occur.

\begin{theorem}\label{thm:main}
There exists an integer $N_0>0$ such that, for every integer $N\ge N_0$ and every $u\in\{1,2,3,4\}$, there is a representation
\begin{equation}\label{eq:repr}
 N=p_1+p_2+\sum_{i=1}^{u}F_{r_i}+\sum_{j=1}^{5-u}L_{s_j},
\end{equation}
where $p_1,p_2$ are primes and all $r_i,s_j$ are positive integers.
\end{theorem}

\begin{corollary}\label{cor:count}
There exists a constant $c>0$ such that every sufficiently large integer $N$ has at least 
$cN(\log N)^3$ ordered representations of the form~\eqref{eq:repr}, for each $u\in\{1,2,3,4\}$. 
The same constant $c$ can be used for all four compositions.
\end{corollary}
An ordered representation records the two prime positions, the $u$ Fibonacci-index positions, and the $5-u$ Lucas-index positions. Distinct indices remain distinct choices when the corresponding recurrence values agree.

The proof separates the recurrence sums according to parity. For $u=1,2,3$,
we start with a positive combination of products containing one Fibonacci
factor and two Lucas factors. Two further factors give the required five
summands. For $u=4$, we use one even Fibonacci term, three odd Fibonacci
terms and one odd Lucas term when $N$ is even. When $N$ is odd, all five
recurrence terms are odd. The remaining target for the two primes is even
in every case.

In the circle method, we integrate products of exponential sums to count
representations. We separate neighborhoods of rational numbers with small
denominators, called the major arcs, from their complement, the minor arcs.
The major-arc formulas used here are due to Pintz~\cite{Pintz} and
Pintz--Ruzsa~\cite{PR}. The uniform bound for pairs of primes with prescribed
difference is recorded by Johnston--Trudgian~\cite[Proposition 2]{JT}, using
Lichtman's work~\cite{Lichtman}. We also use Sun's estimates for recurrence
values that agree modulo an integer, including his mean bound $2.1985$ for
odd Fibonacci differences~\cite[Proposition 4.1]{Sun}. That bound is cited
rather than recomputed. The classical period theory is due to
Wall~\cite{Wall}. Our choices of positive recurrence weights, the accompanying
finite inequalities, and the estimate for odd Lucas differences give the
constants required for five mixed summands. We also extend Sun's
exceptional-modulus argument from odd squarefree moduli to general growing
moduli in Lemma~\ref{lem:exc}.

The pointwise prime estimate of Maynard--Pandey--Radziwi\l\l~\cite[Theorem 1.1]{MPR}
suffices for $u=1,2,3$. For $u=4$, we combine their factor decomposition with
Chen--Gupta--Li's estimates for finite character sums~\cite[Lemma 9.2 and
Proposition 10.2]{CGL} to derive the fourth-moment bound in
Proposition~\ref{prop:high4}. Here a fixed smooth weight is applied to the prime variable, as defined in
Section~\ref{sec:defs}. The relation between fourth moments and counts of
quadruples of characters and real parameters is already used
in~\cite[Lemma 9.2]{CGL}.
We apply it to the prime sum with the conductor factors and integration range
needed here. Kumchev's theorem~\cite[Theorem 2]{Kum} handles the remaining
minor arcs. The definitions and the deduction are given in Section~\ref{sec:proofs}.

The proof is computer-assisted in the determination of explicit bounds for
its constants and the exponents measuring how small the large-value sets
are (the large-deviation exponents). Their ranges matter: the
major-arc lower bound must exceed the minor-arc upper bound, and the set on
which the recurrence sums are large must have sufficiently small measure.
These conditions are proved by exact integer and rational comparisons, with
all truncation and rounding errors bounded in the required direction. The
analytic implications, infinite tails and limiting arguments are proved in
the text. Thus the computation supplies rigorous finite steps of the proof,
not experimental evidence for the theorem. Appendix~\ref{app:cert} describes the data and checking programs supplied
with the paper.

Section~\ref{sec:defs} gives the notation and the estimates used in the main proof. Section~\ref{sec:mainpf} derives the theorem and the counting result. Section~\ref{sec:proofs} proves the analytic and large-deviation estimates. The recurrence-difference averages and the finite verification details are given in the appendices.

\section{Notation and the main estimates}\label{sec:defs}
\subsection{Recurrence sums and parity}
Write $e(t)=\exp(2\pi it)$ and $\T=\mathbb R/\mathbb Z$. Integrals on $\T$
use Lebesgue measure of total mass one; $\meas E$ denotes that measure and
$|t|_\T=\min_{k\in\Z}|t-k|$ denotes distance to the nearest integer.
For a finite set, $|E|$ or $\#E$ denotes its number of elements.
The symbol $\ind_E$ is the indicator of a set or condition: it is one when
that condition holds and zero otherwise. The norms $\|f\|_2$ and
$\|f\|_\infty$ denote the usual $L^2$ and essential-supremum norms, on $\T$
unless another domain is specified. For integers, $(a,b)$ denotes their
greatest common divisor, and $\lcm(a,b)$ their least common multiple.
A positive integer is squarefree if no prime square divides it.

Let $\mu$ be the M\"obius function: it is zero off the squarefree integers,
and equals $(-1)^k$ on a product of $k$ distinct primes. Euler's function
$\varphi(q)$ counts the residue classes coprime to $q$. The von Mangoldt
function $\Lambda$ equals $\log p$ at each prime power $p^k$, $k\ge1$, and
is zero otherwise. We use $f\ll g$, or $f=O(g)$, for $|f|\le Cg$ with an
allowed constant $C$, and $f=o(g)$ for $f/g\to0$. For positive quantities,
$f\asymp g$ means both $f\ll g$ and $g\ll f$. Put
\[
 \phi=\frac{1+\sqrt5}{2},\qquad \psi=\frac{1-\sqrt5}{2},\qquad
 \theta=\frac{1999}{2000}.
\]
For sufficiently large $N$, choose $J=3M$ to be the largest positive multiple of three for which $L_J\le N^\theta/100$. The identities $F_n=(\phi^n-\psi^n)/\sqrt5$ and $L_n=\phi^n+\psi^n$, known as Binet's formulas, give
\begin{equation*}
 J=\frac{\theta\log N}{\log\phi}+O(1).
\end{equation*}
Define the four exponential sums directly by
\begin{align*}
 A(\alpha)&=\sum_{\substack{1\le j\le J\\j\equiv2\ (3)}}e(F_{j+1}\alpha),&
 B(\alpha)&=\sum_{\substack{1\le j\le J\\j\not\equiv2\ (3)}}e(F_{j+1}\alpha),\\
 C(\alpha)&=\sum_{\substack{1\le j\le J\\3\mid j}}e(L_j\alpha),&
 D(\alpha)&=\sum_{\substack{1\le j\le J\\3\nmid j}}e(L_j\alpha),
\end{align*}
where congruences marked $(3)$ are modulo three,
and normalize them by
\[
 a=A/M,\qquad b=B/(2M),\qquad c=C/M,\qquad d=D/(2M).
\]
The terms in $a,c$ are even; those in $b,d$ are odd. All four normalized sums have modulus at most one. The Fibonacci values in each selected subsequence are distinct, because the first Fibonacci index used is $2$.

For the vector of absolute values $\boldsymbol x=(x_a,x_b,x_c,x_d)=(|a|,|b|,|c|,|d|)$, set
\begin{equation*}
 \begin{aligned}
 P_e(a,b,c,d)&=\frac3{10}ad^2+\frac7{10}bcd,\\
 P_o(a,b,c,d)&=\frac12bd^2+\frac3{10}acd+\frac15bc^2,
 \end{aligned}
\end{equation*}
and let $T_e(\boldsymbol x),T_o(\boldsymbol x)$ be the same polynomials evaluated at the nonnegative real variables $\boldsymbol x$. Here and below $\nu=e,o$ means even or odd. Thus $|P_\nu|\le T_\nu\le1$. For $u=1,2,3$, define
\begin{equation}\label{eq:Wu}
 W_{u,\nu}=b^{u-1}d^{3-u}P_\nu(a,b,c,d),\qquad \nu=e,o.
\end{equation}
Every monomial has $u$ Fibonacci factors and $5-u$ Lucas factors. The subscript $\nu$ denotes the parity of its recurrence sum. The two extra factors in~\eqref{eq:Wu} select odd terms, so their total is even.

For $u=4$, use
\begin{equation*}
 W_{4,e}=ab^3d,\qquad W_{4,o}=b^4d.
\end{equation*}
In the finite expansion $W_{u,\nu}(\alpha)=\sum_R w_R e(R\alpha)$, all Fourier coefficients $w_R$, with $u,\nu,J$ fixed, are nonnegative and $\sum_Rw_R=1$. Its support is the set of integers $R$ for which $w_R\ne0$. Its largest Fibonacci index is $J+1$, and $F_{J+1}\le L_J$. Hence every recurrence sum $R$ in its support satisfies
\begin{equation}\label{eq:Rsize}
 0<R\le5L_J\le N^\theta/20=o(N).
\end{equation}
If $\nu$ is the parity of $N$, then $N-R$ is a positive even integer in $[N/2,N]$ for all sufficiently large $N$.

Every terminating decimal in this paper is an exact rational number. Implied constants may depend on parameters fixed before $N$ tends to infinity. The notation $N^{o(1)}$ denotes a factor bounded by $N^\varepsilon$ for every fixed $\varepsilon>0$, once $N$ is sufficiently large. Fixed small power losses are recorded separately.

\subsection{Major arcs and quadratic estimates}
Let
\[
 S(\alpha)=\sum_{N^{0.99}<p\le N}(\log p)e(p\alpha),
\]
and define
\[
 \MA(P)=\bigcup_{\substack{q\le P\\(a_0,q)=1}}
 \left\{\alpha\in\T:\left|\alpha-\frac{a_0}{q}\right|_{\T}\le\frac{P}{qN}\right\}.
\]
The numerator $a_0$ is unrelated to the recurrence sum $a$. A common truncation will be selected in the range
\begin{equation}\label{eq:arcs}
 N^{0.439}\le P_+\le N^{0.441},\qquad
 \MA_+=\MA(P_+),\qquad\m=\T\setminus\MA_+.
\end{equation}
Put
\[
 E_B=5.697364,\qquad E_D=5.832459995.
\]
\begin{lemma}\label{lem:prime}
The truncation in~\eqref{eq:arcs} can be chosen so that
\[
 \sup_{\alpha\in\m}|S(\alpha)|\le N^{19/24+o(1)},
\]
and
\begin{equation}\label{eq:energy}
 \int_\m|Sb|^2\dd\alpha\le(E_B+o(1))N,\qquad
 \int_\m|Sd|^2\dd\alpha\le(E_D+o(1))N.
\end{equation}
\end{lemma}
\begin{lemma}\label{lem:major}
For the same truncation, let $\nu$ be the parity of $N$. For $u=1,2,3$,
\[
 \int_{\MA_+}S(\alpha)^2W_{u,\nu}(\alpha)e(-N\alpha)\dd\alpha
 \ge(1.99426-o(1))N.
\]
\end{lemma}
\begin{lemma}\label{lem:large}
Set
\[
 \lambda=\frac{4274}{12500}=0.34192,\qquad
 \mathcal E_\nu=\{\alpha:T_\nu(\boldsymbol x(\alpha))>\lambda\}.
\]
Then
\[
 \meas(\mathcal E_e\cup\mathcal E_o)\ll N^{-11667/20000+o(1)}.
\]
In particular, the exponent exceeds $7/12$ by $1/60000$.
\end{lemma}

\subsection{Estimates for the fourth composition}
Set $\lambda_4=0.3499$ and
\[
 \mathcal E_{4,e}=\{|abd|>\lambda_4\},\qquad
 \mathcal E_{4,o}=\{|b|^2|d|>\lambda_4\}.
\]
These sets differ from $\mathcal E_e,\mathcal E_o$. Write
\[
 M_{4,e}=1.99405,\quad M_{4,o}=1.9939,\qquad
 \delta_{4,e}=0.56217,\quad\delta_{4,o}=0.5637.
\]
\begin{lemma}\label{lem:fourdat}
If $\nu$ is the parity of $N$, then
\begin{align}
 \int_{\MA_+}S(\alpha)^2W_{4,\nu}(\alpha)e(-N\alpha)\dd\alpha
 &\ge(M_{4,\nu}-o(1))N,\label{eq:fourmaj}\\
 \meas\mathcal E_{4,\nu}&\ll N^{-\delta_{4,\nu}+o(1)}.\label{eq:fourmeas}
\end{align}
\end{lemma}

Fix a sufficiently small $\tau>0$. Choose an infinitely differentiable,
compactly supported function $w\in C_c^\infty(\mathbb R)$ with $0\le w\le1$, support in $[\tau,1]$, and $w=1$ on $[2\tau,1-2\tau]$. Put
\[
 S_w(\alpha)=\sum_p(\log p)w(p/N)e(p\alpha).
\]
\begin{lemma}\label{lem:smooth}
For every measurable $G\subseteq\T$ and $\nu=e,o$,
\begin{align*}
 \left|\int_G(|Sb|^2-|S_wb|^2)\dd\alpha\right|&\ll\sqrt\tau\,N+o(N),\\
 \left|\int_G(S^2-S_w^2)W_{4,\nu}e(-N\alpha)\dd\alpha\right|
 &\ll\sqrt\tau\,N+o(N).
\end{align*}
The constants multiplying $\sqrt\tau$ are independent of the derivatives of $w$ and can be chosen uniformly for the indicated class of weights.
\end{lemma}
For every $\alpha$, choose a reduced Dirichlet approximation, using a fixed ordering of the candidates,
\[
 \alpha=a_0/q+\beta,\qquad q\le\sqrt N,\quad
 |\beta|\le(q\sqrt N)^{-1}.
\]
The selected sets are measurable. Define
\[
 B_N=\max(q,qN|\beta|),\qquad
 \m_{\rm lo}=\{\alpha\in\m:B_N\le N^{0.46}\},\qquad
 \m_{\rm hi}=\{\alpha\in\m:B_N>N^{0.46}\}.
\]
On $\m$, one has $B_N>P_+$. Otherwise the selected approximation would place $\alpha$ on a major arc.
\begin{lemma}\label{lem:primein}
For the fixed $\tau,w$ above,
\begin{align}
 \sup_{\m_{\rm lo}}|S_w(\alpha)|&\le N^{0.7805+o(1)},\label{eq:lowprime}\\
 \int_{\m_{\rm hi}}|S_w(\alpha)|^4\dd\alpha&\le N^{2.5401+o(1)}.\label{eq:fourth}
\end{align}
These estimates hold for every sufficiently large scale $N$, irrespective of its parity.
\end{lemma}
\begin{lemma}\label{lem:badset}
Let $E_N\subseteq\T$ be measurable and suppose that $\meas E_N\ll N^{-\delta+o(1)}$. Then
\begin{equation}\label{eq:Etrans}
 \int_{\m\cap E_N}|S_w|^2\dd\alpha
 \ll N^{1.561-\delta+o(1)}+N^{(2.5401-\delta)/2+o(1)}.
\end{equation}
In particular, this integral is $o(N)$ if $\delta>0.561$.
\end{lemma}

\section{Proof of the main results}\label{sec:mainpf}
\subsection{The cases \texorpdfstring{$u=1,2,3$}{u=1,2,3}}
\begin{proof}[Proof of Theorem~\ref{thm:main} for $u=1,2,3$]
Let $\nu$ be the parity of $N$ and consider
\[
 \mathcal R_u(N)=\int_\T S(\alpha)^2W_{u,\nu}(\alpha)e(-N\alpha)\dd\alpha.
\]
Fourier orthogonality makes this a nonnegative weighted count of representations with exactly $u$ Fibonacci terms and $5-u$ Lucas terms. The normalization of each recurrence sum and every coefficient of $P_\nu$ are positive.

On $\m\setminus\mathcal E_\nu$, we have $|P_\nu|\le\lambda$. Lemma~\ref{lem:prime} and Cauchy's inequality give an upper bound $(\lambda E_u+o(1))N$ for the absolute contribution of this set, where
\[
 E_u=\begin{cases}
 E_D,&u=1,\\
 5.764517,&u=2,\\
 E_B,&u=3.
 \end{cases}
\]
For $u=2$, use $(5.764517)^2>E_BE_D$. The major-arc integral is real, since its domain is invariant under $\alpha\mapsto-\alpha$.

On $\m\cap\mathcal E_\nu$, use $|W_{u,\nu}|\le1$. Lemmas~\ref{lem:prime} and~\ref{lem:large} give
\[
 \left|\int_{\m\cap\mathcal E_\nu}S^2W_{u,\nu}e(-N\alpha)\dd\alpha\right|
 \ll N^{19/12-11667/20000+o(1)}=N^{1-1/60000+o(1)}=o(N).
\]
The major-arc lower bound of Lemma~\ref{lem:major} exceeds the ordinary minor-arc upper bound in every case:
\begin{center}
\begin{tabular}{crr}
\toprule
$u$ & $\lambda E_u$ & $1.99426-\lambda E_u$\\\midrule
1&1.9942347214904&0.0000252785096\\
2&1.9710036526400&0.0232563473600\\
3&1.9480426988800&0.0462173011200\\\bottomrule
\end{tabular}
\end{center}
All entries are exact terminating decimals. Thus $\mathcal R_u(N)\ge c_uN>0$ for sufficiently large $N$, which proves the three assertions.
\end{proof}

\subsection{Four Fibonacci terms and one Lucas term}
\begin{proof}[Proof of Theorem~\ref{thm:main} for $u=4$]
Choose $\nu$ according to the parity of $N$. In $W_{4,e}=ab^3d$, four of the five recurrence terms are odd, so their sum is even. In $W_{4,o}=b^4d$, all five are odd. Consequently $N-R$ is even in both cases. By~\eqref{eq:Rsize}, it lies in $[N/2,N]$.

The integral
\[
 \mathcal R_{4,w}(N)=\int_\T S_w(\alpha)^2W_{4,\nu}(\alpha)e(-N\alpha)\dd\alpha
\]
is a nonnegative weighted representation count. Lemmas~\ref{lem:fourdat} and~\ref{lem:smooth} give a major-arc contribution of at least
\[
 (M_{4,\nu}-O(\sqrt\tau)-o(1))N.
\]
On $\m\setminus\mathcal E_{4,\nu}$, both weights satisfy $|W_{4,\nu}|\le\lambda_4|b|^2$. Lemmas~\ref{lem:prime} and~\ref{lem:smooth} bound the absolute contribution by
\[
 0.3499(5.697364+O(\sqrt\tau)+o(1))N.
\]
On the exceptional minor arcs, $|W_{4,\nu}|\le1$. Lemmas~\ref{lem:fourdat} and~\ref{lem:badset} give
\[
 \begin{aligned}
 \int_{\m\cap\mathcal E_{4,e}}|S_w|^2\dd\alpha
 &\ll N^{0.99883+o(1)}+N^{0.988965+o(1)}=o(N),\\
 \int_{\m\cap\mathcal E_{4,o}}|S_w|^2\dd\alpha
 &\ll N^{0.9973+o(1)}+N^{0.9882+o(1)}=o(N).
 \end{aligned}
\]
Finally,
\begin{equation}\label{eq:margins}
 \begin{aligned}
 1.99405-0.3499(5.697364)&=0.0005423364,\\
 1.9939-0.3499(5.697364)&=0.0003923364.
 \end{aligned}
\end{equation}
Choose $\tau$ so that the smoothing losses in both cases are less than half the smaller margin. Then fix $w$ and take $N$ sufficiently large. We obtain $\mathcal R_{4,w}(N)\ge c_4N>0$. Positivity gives an actual representation. Taking the maximum of the finitely many thresholds completes Theorem~\ref{thm:main}.
\end{proof}

\subsection{Counting representations}
\begin{proof}[Proof of Corollary~\ref{cor:count}]
Each proof above gives a normalized weighted count bounded below by a positive constant times $N$. An ordered recurrence-index tuple in a fixed monomial has weight at most $C J^{-5}$, because each subsequence has $J/3$ or $2J/3$ indices. The number of monomials is fixed, so the same bound applies to the sum of their weights. The two prime weights have product at most $(\log N)^2$; the smooth factors, when present, lie between zero and one. Thus the number of ordered representations is at least
\[
 c\frac{NJ^5}{(\log N)^2}\gg N(\log N)^3.
\]
For $u=4$, the recurrence normalizations are $M(2M)^4$ for even targets and $(2M)^5$ for odd targets. Both are of order $J^5$. Repetitions introduce no factorial, since the positions and indices are ordered. The minimum of the finitely many positive constants gives one $c$ for all four compositions.
\end{proof}

\section{Proofs of the estimates}\label{sec:proofs}
\subsection{Singular series and recurrence-difference averages}\label{sec:analytic}
For every nonzero integer $h$, define the singular series for prime pairs by
\begin{equation}\label{eq:singular}
 \Sg(h)=2C_2\ind_{2\mid h}\prod_{\substack{p\mid |h|\\p>2}}\frac{p-1}{p-2},
 \qquad C_2=\prod_{p>2}\left(1-\frac1{(p-1)^2}\right).
\end{equation}
Products indexed by $p$ are over primes. We do not define $\Sg(0)$; zero shifts are always treated separately. For odd squarefree $r$, write
\[
 \omega(r)=\prod_{p\mid r}\frac1{p-2},\qquad\omega(1)=1.
\]
Then
\begin{equation}\label{eq:positive}
 \Sg(h)=2C_2\ind_{2\mid h}\sum_{\substack{r\mid h\\r\text{ odd squarefree}}}\omega(r).
\end{equation}
Let $\pi(r)$ be the least positive period of the state $(F_n,F_{n+1})$ modulo $r$, with $\pi(1)=1$. For a prime $p$, let $z(p)$ be the least positive index $h$ with $p\mid F_h$, called the rank of apparition. The period and rank facts used below are quoted in Appendix~\ref{app:fib}. For $U=b$ or $d$, denote its distinct odd recurrence values by $u_i$ and put $K=2M=2J/3$ and
\[
 \mathfrak A_{U,J}=K^{-2}\sum_{i\ne j}\Sg(u_i-u_j).
\]
The two arithmetic estimates needed below are
\begin{equation}\label{eq:coll}
 \limsup_J\mathfrak A_{b,J}<2.1985,
 \qquad \lim_J\mathfrak A_{d,J}<2.255.
\end{equation}
The limits are along multiples of three. The first is the cited result in Appendix~\ref{app:fib}; Appendix~\ref{app:lucas} proves the second, including all omitted moduli. We use their boundedness here before establishing any major-arc lower bound.

\subsection{Exceptional weights for general growing moduli}
The odd squarefree case is Sun's exceptional-modulus lemma
\cite[Lemma 5.1]{Sun}. The following extension also covers even moduli and
unbounded prime powers. We give the full proof because the squarefree kernel
of a growing conductor need not itself grow.

\begin{lemma}\label{lem:exc}
Let $J\asymp\log N$, let $u_n$ be either $F_n$ or $L_n$, and restrict
the indices to a fixed periodic subfamily of $1\le n\le J+O(1)$.
Let $C_0$ be a fixed nonzero integer. For every integer modulus
$q=q(N)$ tending to infinity,
\begin{equation*}
 \sum_{\substack{n\text{ in the subfamily}\\m\ne u_n\\q\mid C_0(m-u_n)}}
 \mathfrak S(m-u_n)=o(J),
\end{equation*}
uniformly for all integers $m$ with $|m|\le2N$. The modulus $q$ need
not be odd or squarefree.
\end{lemma}

\begin{proof}
Replace $q$ by $q/\gcd(q,|C_0|)$, which still tends to infinity. The
condition then becomes $q\mid m-u_n$. Put
$q_0=\prod_{p\mid q,\ p>2}p$, the product of the distinct odd prime divisors.
If $q_0\to\infty$, Sun's Lemma~5.1 applies directly to $q_0$ and to the full
Fibonacci or Lucas family. A restriction to a periodic subfamily only
decreases this nonnegative sum. Changing the Fibonacci index from $n+1$ to
$n$ changes at most a bounded number of terms, each of size $O(\log\log N)=o(J)$.

We need the following estimates from the proof of that lemma for the
remaining case. For odd squarefree $d$, let
$T(d)=\max_{p\mid d}\pi(p)$, with $T(1)=1$. By the occurrence bound in
\cite[Lemma 4.1]{Sun}, and \cite[(5.2) and the proof of Lemma 5.1]{Sun},
\begin{align}
 \#\{1\le n\le J+O(1):d\mid m-u_n\}
 &\ll J/T(d)+1,\label{eq:030}\\
 \sum_{d\text{ odd squarefree}}\frac{\omega(d)}{T(d)}&<\infty.\label{eq:031}
\end{align}
The fixed prime $5$ affects only the implied constant. For
$H=\prod_{1\le n\le J+O(1),\ m\ne u_n}|m-u_n|$, the same proof gives
\begin{equation}\label{eq:033}
 \sum_{\substack{d\mid H\\d\text{ odd squarefree}}}\omega(d)
 \ll (\log\log(3H))^2\ll(\log J)^2=o(J),
\end{equation}
since $\log H=O(J\log N)=O(J^2)$. These estimates are uniform for
$|m|\le2N$, and remain valid after restricting the index set.

It remains to consider a subsequence on which $q_0$ is bounded. On a further
subsequence, all prime divisors of $q$ belong to a fixed finite set, and
$p^a\mid q$ for one fixed prime $p$ with $a\to\infty$.
For a nonzero integer $h$, let $v_p(h)$ denote the largest integer $v\ge0$
for which $p^v\mid h$; this is the $p$-adic valuation.
Lengyel's valuation formulas~\cite{Lengyel}, recorded in
\cite[Theorem 1.4]{MR}, imply $v_p(F_h)\le c_p+v_p(h)$ for $h\ge1$.
Since $F_{2h}=F_hL_h$, the same bound holds for $L_h$ after increasing the
constant by $v_p(2)$. Thus for a constant $c_p$ depending only on $p$,
\begin{equation}\label{me:valuation}
 v_p(F_h)\le c_p+v_p(h),\qquad v_p(L_h)\le c_p+v_p(h)\quad(h\ge1).
\end{equation}

For distinct positive indices of the same parity, interchange the
indices if necessary so that $n>\ell$, and write
$u=(n-\ell)/2$, $v=(n+\ell)/2$. Binet's formula gives
\[
 \begin{aligned}
 F_{v+u}-F_{v-u}
 &=\begin{cases}
 F_uL_v,&u\text{ even},\\
 L_uF_v,&u\text{ odd},
 \end{cases}\\
 L_{v+u}-L_{v-u}
 &=\begin{cases}
 5F_uF_v,&u\text{ even},\\
 L_uL_v,&u\text{ odd}.
 \end{cases}
 \end{aligned}
\]
By \eqref{me:valuation}, both recurrence sequences satisfy
\[
 v_p(u_n-u_\ell)\le c'_p+v_p(n-\ell)+v_p(n+\ell).
\]
In each parity class of indices, choose a solution $n_0$ if one exists.
Every other solution of $u_n\equiv m\pmod{p^a}$ must satisfy
\[
 \begin{gathered}
 n\equiv n_0\text{ or }-n_0\pmod{p^b},\\
 b=\max\{0,\lceil(a-c'_p)/2\rceil\}.
 \end{gathered}
\]
Thus the number of solutions is $O_p(Jp^{-b}+1)$, uniformly in $m$.
Taking the smaller of this estimate and \eqref{eq:030}, we obtain
\[
 \begin{split}
 \#\{n:p^a\mid m-u_n,\ d\mid m-u_n\}
 \ll_p J\min\{p^{-b},T(d)^{-1}\}+1.
 \end{split}
\]
Apply this estimate to the positive expansion~\eqref{eq:positive}. Equation
\eqref{eq:031} and dominated convergence make the main
contribution $o(J)$, and \eqref{eq:033} gives $o(J)$
for the endpoint contribution as well. The two subsequence cases
cover every sequence $q(N)\to\infty$, completing the proof.
\end{proof}

\subsection{Prime differences and transfer between prime intervals}
\begin{lemma}[Johnston--Trudgian {\cite[Proposition 2]{JT}}]\label{lem:JT}
For every fixed $\varepsilon>0$ and all sufficiently large $X$,
\[
 \#\{(p,p'):p,p'\le X,\ p-p'=h\}
 \le(6.7814+\varepsilon)\frac{C_2X}{(\log X)^2}
 \prod_{\substack{\ell\mid h\\\ell>2}}\frac{\ell-1}{\ell-2},
\]
uniformly for positive even $h<X$. The product is over prime divisors of $h$.
\end{lemma}
This counts ordered prime pairs.
Their logarithmic weight is at most $(\log X)^2$, whereas~\eqref{eq:singular} includes the factor $2C_2$.
Taking $0<\varepsilon<0.0006$ gives
\begin{equation}\label{eq:pair}
 \sum_{\substack{p,p'\le X\\p-p'=h}}\log p\log p'
 \le3.391X\Sg(h)\qquad(0<|h|<X,\ 2\mid h).
\end{equation}
Negative shifts are handled by interchanging the primes. For $|h|\ge X$, the sum is zero. At $h=0$, we use the prime number theorem, not the singular series.

Let $S_\delta(\alpha)=\sum_{N^{1-\delta}<p\le N}(\log p)e(p\alpha)$, where $0<\delta\le0.01$ is fixed.
For $U=b$ or $d$, Parseval's identity,~\eqref{eq:pair} and the prime number theorem give
\begin{equation}\label{eq:fullnorm}
 \int_\T|S_\delta U|^2\dd\alpha
 \le\frac{N\log N}{K}+3.391N\mathfrak A_{U,J}+o(N)=O(N),
 \qquad K=2J/3.
\end{equation}
The diagonal contributes $(1+o(1))N\log N/K$. Every off-diagonal shift is nonzero and even. In particular, $\|SU\|_2=O(\sqrt N)$ without a major-arc lower bound.

Put $D_\delta=S-S_\delta$ and $X=N^{1-\delta}=o(N)$. The coefficients of $D_\delta$ are nonnegative, supported on primes at most $X$, and bounded by $\log p$. The same expansion has diagonal contribution $O(X\log X/K)=O(X)$; the other terms are controlled by~\eqref{eq:pair} and~\eqref{eq:coll}. Thus
\begin{equation}\label{eq:transfer}
 \|D_\delta U\|_2^2=O(X)=o(N).
\end{equation}
For any measurable $G\subseteq\T$, Cauchy's inequality yields
\[
 \left|\int_G(|SU|^2-|S_\delta U|^2)\dd\alpha\right|
 \le\|D_\delta U\|_2\bigl(2\|S_\delta U\|_2+\|D_\delta U\|_2\bigr)=o(N).
\]
The recurrence weights satisfy
\[
 |W_{1,\nu}|\le|d|^2,\qquad
 |W_{2,\nu}|\le|bd|\le\tfrac12(|b|^2+|d|^2),\qquad
 |W_{3,\nu}|,|W_{4,\nu}|\le|b|^2.
\]
The same argument gives
\begin{equation}\label{eq:transW}
 \left|\int_G(S^2-S_\delta^2)W_{u,\nu}e(-N\alpha)\dd\alpha\right|=o(N).
\end{equation}
We can therefore apply the two explicit formulas on their permitted prime intervals and transfer the result to $S$. Their lower-cutoff parameters need not coincide, and need not equal $0.01$.

\subsection{A common enlarged major-arc formula}
A Dirichlet character modulo $q$ is a homomorphism from the group of residue
classes coprime to $q$ to the unit circle, extended by zero on other classes.
It is primitive if it is not induced from a proper divisor of $q$; the
modulus of its inducing primitive character is its conductor,
$\operatorname{cond}\chi$. The principal character is one on classes coprime
to its modulus. Write $L(s,\chi)$ for its Dirichlet $L$-function, defined by
$\sum_{n\ge1}\chi(n)n^{-s}$ for $\operatorname{Re}s>1$ and then by analytic
continuation. The Riemann zeta function is $\zeta(s)=L(s,1)$ for the character
of conductor one, and $\Gamma$ denotes Euler's gamma function.
We use Pintz \cite[Theorem 1, Main Lemma 1, (1.21), Lemma 4.18]{Pintz}.
His explicit formula applies to the logarithmically weighted prime sum
over $X^{1-\delta}<p\le X$, where $\delta>0$ is sufficiently small
and fixed, and may depend on the fixed truncation margin and accuracy.
If $P_0\le X^{4/9-\varepsilon}$, there exists
$P\in[P_0X^{-\varepsilon},P_0]$ for which the formula on the same
major-arc set holds simultaneously for all the even targets $m$ needed
here. Take $X=N$, $P_0=N^{0.441}$, and $\varepsilon=0.002$.
Since $0.441<4/9-0.002$, the resulting truncation is $P_+$.
Choose a fixed $0<\delta\le0.01$ permitted by the original theorem,
and denote the corresponding prime sum by $S_\delta$.

Write
\[
 G_P(m)=\int_{\mathfrak M(P_+)}S_\delta(\alpha)^2e(-m\alpha)\dd\alpha.
\]
The term arising from the pole at $1$ in each factor is
$\mathfrak S(m)(m-2N^{1-\delta}+O(1))$. The other terms arise from zeros $\rho_i=\beta_i+i\gamma_i$ of the
relevant Dirichlet $L$-functions in a shrinking strip near the line
$\operatorname{Re}s=1$, satisfying
\[
 \beta_i\ge1-H/\log N,\qquad |\gamma_i|\le\sqrt N,
 \qquad r_i=\operatorname{cond}\chi_i\le P_+.
\]
For fixed $H$, their number, counted with multiplicity, is bounded
by a fixed constant $K_H$. Include the pole of $\zeta$ at $1$ in
the index set, with coefficient $1$, and give each zero coefficient
$-1$, counting multiplicities. Writing $A_i=1$ for the pole and $A_i=-1$ for each occurrence of a zero,
the supplementary terms have the form
\[
 A_iA_j\mathfrak S(\chi_i,\chi_j,m)
 \frac{\Gamma(\rho_i)\Gamma(\rho_j)}{\Gamma(\rho_i+\rho_j)}
 m^{\rho_i+\rho_j-1}.
\]
The error in the explicit formula is
\[
 O_\varepsilon\bigl(\mathfrak S(m)N e^{-c_\varepsilon H}\bigr)
 +O_\varepsilon(N^{1-\delta}).
\]
All targets under consideration are even and satisfy $N/2\le m\le N$.
Here $\mathfrak S(\chi_i,\chi_j,m)$ is Pintz's generalized singular series,
the character-dependent coefficient of the corresponding term in his
explicit formula. We use the bounds in the cited Main Lemma~1:

\begin{align}
 |\mathfrak S(\chi_i,\chi_j,m)|
 &\le\mathfrak S(m),\label{eq:039}\\
 |\mathfrak S(\chi_i,\chi_j,m)|
 &\ll\mathfrak S(m)U^{-1/2}(\log\log(3U))^2,
 \label{eq:040}
\end{align}
where
\[
 \begin{split}
 U=\max\bigg\{&
 \frac{r_i^2}{(r_i,r_j)^2},\frac{r_j^2}{(r_i,r_j)^2},
 \frac{r_i}{(r_i,m)},\\
 &\frac{r_j}{(r_j,m)},\operatorname{cond}(\chi_i\chi_j)\bigg\}.
 \end{split}
\]
Each $r_i$ here is the full primitive conductor, not its squarefree part.

Fix a small accuracy parameter $\epsilon>0$, and then take $H$
sufficiently large. For sufficiently large $N$, all the retained real
parts lie in $[3/4,1]$. We prove directly that
\begin{equation}\label{eq:041}
 \frac{\Gamma(\rho_i)\Gamma(\rho_j)}{\Gamma(\rho_i+\rho_j)}
 \ll(1+\max\{|\gamma_i|,|\gamma_j|\})^{-1/4}.
\end{equation}
Use the beta integral, and put
$Y=\max(1,|\gamma_i|,|\gamma_j|)$. The result is immediate for
bounded $Y$. If $Y\ge4$, remove an interval of length $Y^{-1/2}$
at each endpoint; their integrals are $O(Y^{-3/8})$.
On the remaining interval, the amplitude
$x^{\beta_i-1}(1-x)^{\beta_j-1}$ has supremum and total variation
$O(Y^{1/8})$. The phase is
$\gamma_i\log x+\gamma_j\log(1-x)$.
If the two imaginary parts have the same sign, the absolute value of
the second derivative is at least $Y$. If they have opposite signs,
the absolute value of the first derivative is at least $Y$, and at
most two intervals of monotonicity are required. The first- or
second-derivative estimate for oscillatory integrals, followed by
partial summation, proves \eqref{eq:041}.
Since $|m^{\rho_i+\rho_j-1}|\le N$, a fixed height $U_\gamma$
can be chosen so that all terms containing
$|\gamma_i|>U_\gamma$ or $|\gamma_j|>U_\gamma$ contribute at most
the required small multiple of $\epsilon\mathfrak S(m)N$.
Next, use \eqref{eq:040} to choose a fixed $U_0$ and
discard the terms with $U>U_0$ in the same way.
These estimates are uniform for $N/2\le m\le N$; the fixed bound
$K_H^2$ on the number of terms controls the total error.

For the remaining supplementary terms, put
$C_\epsilon=\operatorname{lcm}(1,\ldots,\lceil U_0\rceil)$.
The inequality $r_i/(r_i,m)\le U_0$ implies
$r_i\mid C_\epsilon m$, and the same holds for $r_j$.
Every term other than this leading term contains at least one zero.
For any fixed conductor bound, there are only finitely many relevant
characters. They have finitely many zeros in the fixed height range,
and none of these zeros has real part $1$. Thus they eventually leave
the shrinking strip $\beta\ge1-H/\log N$.
The minimum conductor of the remaining zeros therefore tends to
infinity, unless there are eventually no remaining zeros.
This argument includes the zeros of the principal character modulo $1$;
it does not require every zero to belong to a nonprincipal character.

Consequently, there is a family of full conductors $r_\nu$, with
bounded cardinality and minimum tending to infinity, such that
\begin{align}
 |G_P(m)|&\ll_\epsilon\mathfrak S(m)N,
 \label{eq:042}\\
 G_P(m)&=\mathfrak S(m)m+O(\epsilon\mathfrak S(m)N)
                 +o(\mathfrak S(m)N).
 \label{eq:043}
\end{align}
The second estimate holds for targets satisfying
$r_\nu\nmid C_\epsilon m$ for every $\nu$; the first holds for all
targets under consideration.

Expand a monomial of any $W_{u,\nu}$, where $\nu$ is the parity of $N$.
There is at least one Lucas factor. Fix the other four recurrence indices,
write $m_0$ for $N$ minus those four terms, and sum over that Lucas index.
By~\eqref{eq:Rsize}, we have $|m_0|\le2N$ and
$m_0-L_s=N-O(N^\theta)\in[N/2,N]$; zero shifts never occur.
Lemma~\ref{lem:exc}, applied to each growing conductor, shows that the
singular-series weight of its exceptional five-tuples is $o(J^5)$,
uniformly in the four fixed indices. The conductor family has bounded
cardinality. After division by the recurrence normalization, which is
of order $J^5$, equation~\eqref{eq:042} bounds the total
absolute contribution of exceptional tuples by $o(N)$.
This includes any negative exceptional terms.

Choose $\epsilon$ sufficiently small, then take $N$ sufficiently large.
Because $m=N-O(N^\theta)$, the generic term retains at least
$0.9999N\Sg(m)$. Consequently, if
$W_{u,\nu}(\alpha)=\sum_R w_R e(R\alpha)$, then
\begin{equation}\label{eq:avgmajor}
 \int_{\MA_+}S_\delta^2W_{u,\nu}e(-N\alpha)\dd\alpha
 \ge 0.9999N\sum_R w_R\Sg(N-R)-o(N).
\end{equation}
Here $w_R\ge0$ and the coefficients sum to one.
The same inequality holds with $S$ in place of $S_\delta$, by~\eqref{eq:transW}.
The finite convolutions in Section~\ref{sec:local} give the required
lower bounds for the average on the right. This argument uses the
parity of $N-R$, not the parity of $N$ itself. The choice of $P_+$
is common to all targets and all four compositions.

\subsection{An unconditional lower bound for the inner major-arc diagonal}

\begin{lemma}\label{me:diagonal}
Let $0<\delta_-\le0.01$ be fixed, put
$S_- =\sum_{N^{1-\delta_-}<p\le N}(\log p)e(p\alpha)$,
and suppose that $N^{0.399}\le P_-\le N^{0.4}$. Then
\begin{equation}\label{eq:046}
 \int_{\mathfrak M(P_-)}|S_-|^2\dd\alpha
 \ge(0.39-o(1))N\log N.
\end{equation}
\end{lemma}

\begin{proof}
Write $P=P_-$ and $S=S_-$, and take $T_0=100$.
Since $N^{1-\delta_-}>P$, every prime in the sum is coprime
to every $q\le P$. The Ramanujan sum identity and Cauchy's
inequality give
\[
 \begin{aligned}
 \sum_{\substack{a\bmod q\\(a,q)=1}}S(a/q+\beta)
 &=\mu(q)S(\beta),\\
 \sum_{\substack{a\bmod q\\(a,q)=1}}|S(a/q+\beta)|^2
 &\ge\frac{\mu(q)^2}{\varphi(q)}|S(\beta)|^2.
 \end{aligned}
\]
If $q\le P/T_0$, then the interval $|\beta|\le T_0/N$ lies in
the corresponding major arc. All defining major arcs are disjoint
because $2P^2<N$. Hence
\begin{equation}\label{eq:047}
 \begin{split}
 \int_{\mathfrak M(P)}|S|^2\dd\alpha
 \ge{}&\left(\sum_{q\le P/T_0}\frac{\mu(q)^2}{\varphi(q)}\right)\\
 &\mathrel{}\times\int_{-T_0/N}^{T_0/N}|S(\beta)|^2\,d\beta.
 \end{split}
\end{equation}

We use the elementary asymptotic
$\sum_{q\le x}\mu(q)^2/\varphi(q)\sim\log x$. Indeed, for $s>0$ its
Dirichlet series is
\[
 \sum_{q\ge1}\frac{\mu(q)^2}{\varphi(q)q^s}
 =\zeta(1+s)\prod_p\left(1+\frac{p^{-s}-p^{-2s}}{p(p-1)}\right).
\]
The second factor has absolutely summable Dirichlet coefficients and their
sum at $s=0$ is one. Convolution with $1/n$, followed by dominated convergence
for the normalized harmonic sums, gives the stated asymptotic.

The prime number theorem and partial summation give
$N^{-1}S(t/N)\to\int_0^1e(tu)\,du$, uniformly for $|t|\le T_0$;
the omitted primes $p\le N^{1-\delta_-}=o(N)$ contribute $o(N)$.
Thus the integral in \eqref{eq:047} is
$(c_{T_0}+o(1))N$, where
\[
 \begin{split}
 c_{T_0}
 &=\int_{-T_0}^{T_0}\left|\int_0^1e(tu)\,du\right|^2dt\\
 &\ge1-\frac2{\pi^2T_0}\ge\frac{449}{450}.
 \end{split}
\]
Here we used Parseval's identity,
$|\int_0^1e(tu)\,du|\le1/(\pi|t|)$, and $\pi>3$.
Finally, $\log P\ge0.399\log N$ and
$0.399(449/450)>0.39$, proving \eqref{eq:046}.
The argument is unconditional and does not use an exact asymptotic
formula for the diagonal contribution.
\end{proof}

\subsection{Nonzero shifts and ordinary minor-arc energy}
Apply Pintz--Ruzsa~\cite[Theorem 2 and Corollary 1]{PR} with accuracy $10^{-4}$, initial truncation $N^{0.4}$ and margin $0.001$. Choose a fixed $0<\delta_-\le0.01$ allowed by their theorem. There is a common truncation $P_-\in[N^{0.399},N^{0.4}]$. Write
\[
 S_-(\alpha)=\sum_{N^{1-\delta_-}<p\le N}(\log p)e(p\alpha),\qquad
 D_-(h)=\int_{\MA(P_-)}|S_-|^2e(-h\alpha)\dd\alpha.
\]
For nonzero even $|h|\le10^{-4}N/4$, one has $|D_-(h)|\ll\Sg(h)N$.
There is a bounded family of full primitive conductors $r_\nu$, whose minimum tends to infinity, and a fixed integer $C_-$ such that
\[
 |D_-(h)-\Sg(h)N|\le10^{-4}\Sg(h)N
\]
whenever $r_\nu\nmid C_-h$ for all $\nu$. Negative shifts follow by conjugation.
The cited major arcs have the same width $P_-/(qN)$ as ours. All shifts from $b$ or $d$ are $O(N^\theta)$, so they eventually satisfy the required bound on $|h|$.

For each exceptional conductor, fix one index $i$, take $m=u_i$ in Lemma~\ref{lem:exc}, and sum over the other index. The total singular-series weight of the exceptional shifts is $o(K^2)$. Bound their possible negative contributions in absolute value. Lemma~\ref{me:diagonal} now gives, for $U=b,d$,
\[
 \int_{\MA(P_-)}|S_-U|^2\dd\alpha
 \ge(0.39-o(1))\frac{N\log N}{K}
       +0.9999N\mathfrak A_{U,J}-o(N).
\]
Since $P_-<P_+$, the two major-arc sets are nested. Subtract this estimate from~\eqref{eq:fullnorm}, keeping the same finite quantity $\mathfrak A_{U,J}$ on both sides. We obtain
\[
 \int_\m|S_-U|^2\dd\alpha
 \le\left(2.3911\mathfrak A_{U,J}+0.61\frac{\log N}{K}+o(1)\right)N.
\]
Only at this point do we use~\eqref{eq:coll}. Since $\log N/K\to3\log\phi/(2\theta)$ and $\log\phi/\theta<0.481453$,
\[
 \begin{aligned}
 2.3911(2.1985)+0.61\tfrac32(0.481453)&=5.697362845<E_B,\\
 2.3911(2.255)+0.61\tfrac32(0.481453)&=5.832459995=E_D.
 \end{aligned}
\]
Equation~\eqref{eq:transfer} replaces $S_-$ by $S$. This proves~\eqref{eq:energy}.

\subsection{The pointwise estimate for the first three compositions}
\begin{lemma}[Maynard--Pandey--Radziwi\l\l\ {\cite[Theorem 1.1]{MPR}}]\label{lem:MPR}
For a reduced approximation $\alpha=a_0/q+\beta$ with $q\le\sqrt X$ and
$|\beta|\le(q\sqrt X)^{-1}$, one has
\[
 \left|\sum_{n<X}\Lambda(n)e(n\alpha)\right|
 \le X^{o(1)}\left(\frac X{\sqrt{B_X}}+X^{19/24}\right),\qquad
 B_X=\max(q,qX|\beta|),
\]
as $X\to\infty$.
\end{lemma}
Choose the approximation independently at $X=N$ and $X=N^{0.99}$.
If $\alpha\in\m$, each satisfies $B_X>(X/N)P_+$: this is immediate when $q>P_+$; otherwise the minor-arc condition gives $|\beta|>P_+/(qN)$.
The first terms at the two scales are at most $N^{0.7805+o(1)}$ and $N^{0.7755+o(1)}$, respectively.
Subtract the two prefix sums and remove the higher prime powers, whose total logarithmic weight is $O(N^{1/2}\log^2N)$. Endpoint conventions change only $O(\log N)$.
It follows that $\sup_\m|S|\le N^{19/24+o(1)}$, completing Lemma~\ref{lem:prime}.

\subsection{Finite positive convolutions for the major arcs}\label{sec:local}
For odd squarefree $r\le300$, let $T_r=\pi(2r)$. This period is divisible by three. In a complete block of length $T_r$, record the four vectors counting the occurrences of each residue for $A,B,C,D$ modulo $2r$. Such a counting vector is called a histogram; its total mass is the sum of its entries. Divide each vector by that mass. Denote the resulting probability vectors by $\mu_{A,r},\mu_{B,r},\mu_{C,r},\mu_{D,r}$.
For a vector $\boldsymbol k=(k_A,k_B,k_C,k_D)$ of nonnegative integers with sum five, use cyclic convolution $(f*g)(v)=\sum_{x\bmod2r}f(x)g(v-x)$ on $\Z/(2r)\Z$, and define
\[
 \mu_{\boldsymbol k}(r)=
 \min_{\substack{v\bmod2r\\v\equiv k_B+k_D\ (2)}}
 (\mu_{A,r}^{*k_A}*\mu_{B,r}^{*k_B}*\mu_{C,r}^{*k_C}*\mu_{D,r}^{*k_D})(v).
\]
A zeroth convolution power is the unit mass at zero. Put
\begin{equation}\label{eq:Mk}
 M(\boldsymbol k)=0.9999(1.3202)
 \sum_{\substack{r\le300\\r\text{ odd squarefree}}}\omega(r)\mu_{\boldsymbol k}(r).
\end{equation}
Appendix~\ref{app:fib} proves $2C_2>1.3202$.

Every entry of a normalized finite-window histogram differs from its complete-period value by $O_r(J^{-1})$. The same holds for a fixed fivefold convolution. Since $N-R$ is even, the condition $r\mid N-R$ is equivalent to $R\equiv N\pmod{2r}$.
Keeping the terms $r\le300$ in the positive divisor expansion~\eqref{eq:positive} gives a lower bound; no estimate for the discarded terms is needed. Equation~\eqref{eq:avgmajor} therefore yields
\[
 \int_{\MA_+}S^2a^{k_A}b^{k_B}c^{k_C}d^{k_D}e(-N\alpha)\dd\alpha
 \ge(M(\boldsymbol k)-o(1))N
\]
when $N\equiv k_B+k_D\pmod2$.
The factor $0.9999$ occurs only once in~\eqref{eq:Mk}.

For the positive combinations~\eqref{eq:Wu}, take the corresponding combinations of $M(\boldsymbol k)$. The exact rational calculations give the following values, displayed in decimal form:
\begin{center}\small
\begin{tabular}{crr}
\toprule
$u$ & Even targets & Odd targets\\\midrule
1&1.9943102357906828\ldots&1.9942681313902129\ldots\\
2&1.9947047951726578\ldots&1.9948362739612160\ldots\\
3&1.9946296658316858\ldots&1.9945296241776596\ldots\\\bottomrule
\end{tabular}
\end{center}
All six rational numbers exceed $99713/50000=1.99426$.
The integer lower sums in Section~\ref{sec:handmaj}, rather than the displayed digits, establish the strict comparisons. This proves Lemma~\ref{lem:major}.
There are $122$ moduli. At each convolution the verification program checks that the mass is the product of the input masses. All local convolutions now use arbitrary-precision integers and the parity compression proved next.

\subsection{Inspecting the local calculation}\label{sec:handmaj}
The finite calculation has a useful elementary reduction. If $h_i$ is a
histogram modulo $2r$ supported on parity $\epsilon_i$, put
$\widetilde h_i(j)=h_i(2j+\epsilon_i)$ for $j\bmod r$.
Let $s=\epsilon_1+\cdots+\epsilon_5$ and $\nu\equiv s\pmod2$,
$\nu\in\{0,1\}$. Directly adding the five residue representatives gives
\[
 (h_1*\cdots*h_5)(2j+\nu)
 =(\widetilde h_1*\cdots*\widetilde h_5)
       \left(j-\frac{s-\nu}{2}\right).
\]
The shift permutes the residues modulo $r$. Hence minima and total masses
can be calculated on arrays of length $r$, without changing any lower bound.
This applies to every monomial in the eight weights, including the factors
from $C$.

Two examples make the normalization explicit. For $r=3$, one period modulo
$6$ has length $24$ and the unnormalized histograms are
\[
\begin{array}{c|rrrrrr}
 v&0&1&2&3&4&5\\\hline
 A_3(v)&2&0&3&0&3&0\\
 B_3(v)&0&6&0&4&0&6\\
 D_3(v)&0&6&0&4&0&6
\end{array}
\]
In $\Z[X]/(X^6-1)$, put $t=X^2$ and $U=1+t+t^2$. Then
$A_3=3U-1$ and $B_3=D_3=X(6U-2t)$.
The relations $tU=U$, $U^2=3U$ reduce both products to two terms:
\[
 A_3B_3^3D_3=174768U-16,
 \qquad B_3^4D_3=349536XU-32X^3.
\]
Indeed, their terms not containing $U$ are $-16$ and $-32X^3$;
evaluation at $X=1$ gives the masses $8\cdot16^4=524288$ and
$16^5=1048576$, which determine the remaining coefficients.
The minimum probabilities are therefore
$174752/524288$ for even targets and $349504/1048576$ for odd targets.
For $r=5$, $A_5$ equals $4$ on each even residue modulo $10$ and vanishes
on the odd residues. Convolution with it is uniform on the required parity
class. Thus the even $4F+L$ monomial has mass $20\cdot40^4=51200000$
and minimum coefficient $10240000$, giving probability $1/5$.
These are complete calculations of the indicated rows, not statistical tests.

For the final sum, a long common denominator is unnecessary. Write
$c_{u,\nu}(\boldsymbol k)$ for the coefficient of
$a^{k_A}b^{k_B}c^{k_C}d^{k_D}$ in $W_{u,\nu}$, and set
\begin{equation}\label{eq:majorint}
 K_{u,\nu}(r)=\left\lfloor10^{10}(0.9999)(1.3202)\omega(r)
    \sum_{\boldsymbol k}c_{u,\nu}(\boldsymbol k)
                      \mu_{\boldsymbol k}(r)\right\rfloor.
\end{equation}
The lower bound for the average is at least $10^{-10}\sum_rK_{u,\nu}(r)$.
Every rounding is downwards. Each $K$ is checked by multiplying positive
integers on the two sides of its defining floor inequality. The individual
rows are in \path{major_rows.csv}. Grouping them into
$I_j=(50j,50j+50]\cap\Z$, $0\le j\le5$, gives the following short sums.
The numbers of retained moduli are $20,21,20,20,21,20$.

\begin{center}\scriptsize
\setlength{\tabcolsep}{3pt}
\begin{tabular}{crrrrrrr}
\toprule
$(u,\nu)$ & $I_0$ & $I_1$ & $I_2$ & $I_3$ & $I_4$ & $I_5$ & Total\\\midrule
$(1,e)$ & 19772654878 & 102929222 & 34030488 & 16169839 & 10298526 & 7019346 & 19943102299\\
$(1,o)$ & 19772411271 & 102780899 & 34031365 & 16149007 & 10293446 & 7015267 & 19942681255\\
$(2,e)$ & 19775913090 & 103132466 & 34326972 & 16287105 & 10356263 & 7031993 & 19947047889\\
$(2,o)$ & 19776838676 & 103468907 & 34350521 & 16279471 & 10371774 & 7053331 & 19948362680\\
$(3,e)$ & 19775720533 & 102946638 & 34141075 & 16141131 & 10326556 & 7020671 & 19946296604\\
$(3,o)$ & 19775095181 & 102643203 & 34137939 & 16102900 & 10308743 & 7008219 & 19945296185\\
$(4,e)$ & 19771866871 & 101903027 & 33681720 & 15944933 & 10234854 & 6964471 & 19940595876\\
$(4,o)$ & 19770793583 & 101565906 & 33630745 & 15862875 & 10208549 & 6940755 & 19939002413\\
\bottomrule
\end{tabular}
\end{center}
The total in each of the first six rows exceeds $19942600000$; the last
two exceed $19940500000$ and $19939000000$, respectively. Dividing by
$10^{10}$ proves exactly the conservative constants used in the main proof.
No estimate for the omitted positive divisor terms is needed.

\subsection{Counting the zero relations}\label{sec:carry}
Expanding the exponential moments requires the zero relations
\begin{equation}\label{eq:zero}
 \sum_{j=1}^J(r_jF_{j+1}+s_jL_j)=0,\qquad |r_j|,|s_j|\le3.
\end{equation}
Since $L_j=2F_{j+1}-F_j$, these relations can be read from the largest index downwards, starting at $(m,n)=(0,0)$, with the transition
\begin{equation}\label{eq:carry}
 (m,n)\longmapsto(m+n+r+2s,m-s).
\end{equation}
Induction using the Fibonacci recurrence shows that the final value of $m+n$ is the sum in~\eqref{eq:zero}. A path is the sequence of state pairs produced by the transitions. It is
accepted exactly when its final pair satisfies $m+n=0$, not necessarily
$m=n=0$.

Put $z=m\phi+n$ and $z^*=m\psi+n$. The transition becomes
\[
 z\longmapsto\phi z+r\phi+s\sqrt5,\qquad
 z^*\longmapsto\psi z^*+r\psi-s\sqrt5.
\]
At every intermediate step,
\[
 |z^*|\le\frac{3(|\psi|+\sqrt5)}{1-|\psi|}=3(4\phi+1).
\]
At an accepted endpoint, $n=-m$, so $z=-\phi^{-2}z^*$. Backwards in the real direction,
\[
 |z_{i-1}|\le\phi^{-1}|z_i|+\frac{3(3\phi-1)}{\phi}.
\]
The fixed bound for this affine inequality is $3(2\phi+3)$, which exceeds the endpoint bound. Thus every accepted path has $|z|\le3(2\phi+3)$. Since
\[
 m=\frac{z-z^*}{\sqrt5},\qquad n=\frac{\phi z^*-\psi z}{\sqrt5},
\]
we obtain
\[
 |m|\le3\left(3+\frac7{\sqrt5}\right)<19,\qquad
 |n|\le3\left(4+\frac7{\sqrt5}\right)<22.
\]
Consequently every accepted path stays in
\begin{equation}\label{eq:box}
 \mathcal S=([-18,18]\times[-21,21])\cap\Z^2,\qquad |\mathcal S|=1591.
\end{equation}
Deleting transitions that leave this box loses no zero relation.
\subsection{Restriction and symmetry of the states}\label{sec:reduce}
The two bounds used to prove~\eqref{eq:box} give a smaller set:
\[
 \mathcal S_0=\{(m,n)\in\mathcal S:
 |m\phi+n|\le3(2\phi+3),\quad
 |m\psi+n|\le3(4\phi+1)\}.
\]
Every state of an accepted path lies in $\mathcal S_0$, not merely in the
rectangle $\mathcal S$. Both sets are invariant under $(m,n)\mapsto(-m,-n)$.
For $m\ge0$, the exact integer rows of $\mathcal S_0$ are
\begin{center}\small
\begin{tabular}{r rr@{\qquad}r rr@{\qquad}r rr}
\toprule
$m$&$n_{\min}$&$n_{\max}$&$m$&$n_{\min}$&$n_{\max}$&$m$&$n_{\min}$&$n_{\max}$\\\midrule
0&-18&18&7&-18&7&14&-13&-4\\
1&-20&17&8&-17&5&15&-13&-6\\
2&-21&15&9&-16&4&16&-12&-8\\
3&-20&13&10&-16&2&17&-11&-9\\
4&-19&12&11&-15&0&18&-11&-11\\
5&-19&10&12&-15&-1&&&\\
6&-18&9&13&-14&-3&&&\\\bottomrule
\end{tabular}
\end{center}
The negative rows follow by reflection. Here is a direct check of the table.
The four strip conditions are $a+b\sqrt5\le0$ for
\[
\begin{aligned}
 (a,b)&=(m+2n-24,m-6),\quad(-m-2n-24,-m-6),\\
 &\phantom{{}={}}(m+2n-18,-m-12),\quad(-m-2n-18,m-12).
\end{aligned}
\]
For integers $a,b$, the test is equivalent to
\[
 \begin{cases}
 a\le0,&b=0,\\
 a<0\text{ and }a^2\ge5b^2,&b>0,\\
 a\le0\text{ or }a^2\le5b^2,&b<0.
 \end{cases}
\]
For each of the nineteen rows it therefore suffices to test the two displayed
endpoints and exclude their adjacent integers; an adjacent point outside the
rectangle is already excluded. The intersection is an interval, since the
four inequalities are affine in $n$. The positive rows have $358$ points in
all. Thus $|\mathcal S_0|=37+2(358)=753$.

Let $\mathcal S_+$ contain the representatives with $m>0$, together with
$(0,n)$ for $0\le n\le18$. Then $|\mathcal S_+|=377$. Write
$\operatorname{rep}(m,n)$ for the chosen representative of $\{(m,n),(-m,-n)\}$.
For any nonnegative, even coefficient sequences $c_r(f),c_s(l)$ supported on
$-3,\ldots,3$, define a matrix on these $377$ states by its action on a column:
\begin{equation}\label{eq:redact}
 (M(f,l)x)(m,n)=
 \sum_{\substack{|r|,|s|\le3\\(m+n+r+2s,m-s)\in\mathcal S_0}}
 c_r(f)c_s(l)\,
 x_{\operatorname{rep}(m+n+r+2s,m-s)}.
\end{equation}
Rows are indexed by source states. Contributions to the same destination
orbit are added; they are not averaged.

\begin{lemma}\label{lem:reduce}
The matrices~\eqref{eq:redact}, starting at $(0,0)$ and accepting at
$(m,-m)$ for $0\le m\le13$, give the same weighted sum of accepted paths
as the full rectangular model. The terminal function is one on these
fourteen representatives and zero elsewhere. No orbit-size factor is needed.
\end{lemma}
\begin{proof}
Restriction to $\mathcal S_0$ loses no accepted path, by the two strip bounds.
The transition satisfies
\[
 T_{-r,-s}(-m,-n)=-T_{r,s}(m,n).
\]
Its weight is unchanged because both coefficient sequences are even.
Consequently every restricted single-step matrix preserves symmetric functions.
If $I$ denotes symmetric lifting from $\mathcal S_+$ to $\mathcal S_0$ and
$K$ is the restricted matrix, then $KI=IM(f,l)$, directly from the sum
in~\eqref{eq:redact}. The identity holds for any ordered product of steps.
The initial state is fixed by reflection. The full terminal indicator is
symmetric, so evaluation at the initial state gives exactly the quotient
path sum. At a terminal state $n=-m$, the second strip gives
$|m|\le12+3/\phi<14$. This gives precisely the fourteen representatives stated.
The quotient transition already sums the paths reaching either member of an
orbit, which explains the absence of an additional terminal multiplicity.
\end{proof}

The reduction is independent of the moment parameters and applies to all
$34$ certificates below. It also explains how a positive-vector inequality on the larger
box can be reduced: replace $v$ by $v(m,n)+v(-m,-n)$ and restrict it.
If $R$ is restriction and $E_0$ is extension by zero, positivity gives
\[
 (RK_1E_0)(RK_2E_0)(RK_3E_0)Rv
 \le R K_1K_2K_3v.
\]
Thus a strict upper-vector inequality is preserved. The smaller vectors
supplied here are checked directly on $\mathcal S_+$; neither this construction
nor the original $1591$-coordinate files are needed by their checker.

\subsection{Exponential moments from nonnegative matrices}
For $t\ge0$ and $k\ge0$, define the modified Bessel coefficients by
\[
 I_k(t)=\sum_{j\ge0}\frac{(t/2)^{2j+k}}{j!(j+k)!}.
\]
Let
\[
 c_r(t)=I_{|r|}(t)\quad(0<|r|\le3),\qquad
 c_0(t)=e^t-2\sum_{r=1}^3I_r(t),\qquad
 B_t(v)=\sum_{r=-3}^3c_r(t)e^{irv}.
\]
The Fourier expansion of $e^{t\cos v}$ gives
\[
 B_t(v)-e^{t\cos v}=2\sum_{r>3}I_r(t)(1-\cos rv)\ge0.
\]
All coefficients are nonnegative and $B_t(v)>0$. This is the positive trigonometric majorant used in~\cite[Section 7.1]{Sun}.

Use these even coefficients in the reduced matrix $M(f,l)$ on $\mathcal S_+$ defined by~\eqref{eq:redact}. Lemma~\ref{lem:reduce} preserves the accepted path sum. For $\boldsymbol t=(t_a,t_b,t_c,t_d)\ge0$, put
\[
 \mathcal M(\boldsymbol t)=M(t_a,t_d)M(t_b,t_d)M(t_b,t_c).
\]
The actual reverse reading starts with a cyclic permutation of these three matrices. If they are denoted by $X,Y,Z$, then
$(ZXY)^M=Z(XYZ)^{M-1}XY$. Fixed endpoint matrices change only an implied constant.

A finite certificate for the following matrix bound is an explicitly supplied positive integer vector together with exact inequalities for its coordinates. Suppose that a strictly positive vector $v$ satisfies $\mathcal M(\boldsymbol t)v<\rho v$.
Then $\mathcal M(\boldsymbol t)^kv\le\rho^kv$.
The acceptance indicator and its image under the fixed endpoint matrices are bounded by a fixed multiple of $v$. Apply the trigonometric majorant at each index, expand, and integrate. Only the zero relations survive. Their phase factors have modulus one, so taking absolute values gives the nonnegative matrix sum. We conclude, uniformly for four real phases, that
\begin{equation}\label{eq:moment}
 \int_\T\exp\!\left(\sum_{U\in\{A,B,C,D\}}
 t_U\operatorname{Re}(e^{-i\gamma_U}U(\alpha))\right)\dd\alpha
 \ll\rho^{J/3}.
\end{equation}
Here $t_A=t_a$, and similarly for the other subscripts. Upward rounding is applied only to the nonnegative sum after the phase factors have been removed. It is not used as a pointwise comparison of trigonometric polynomials.

Define
\[
 \ell_{\boldsymbol t}(\boldsymbol x)=
 \frac{t_ax_a+2t_bx_b+t_cx_c+2t_dx_d}{3}.
\]
Choose a grid of $2000$ equally spaced values for each phase. If $\ell_{\boldsymbol t}(\boldsymbol x)>H$, some grid choice makes the exponent in~\eqref{eq:moment} exceed $JH\cos(\pi/2000)$. Markov's inequality and the fixed union bound give
\begin{equation}\label{eq:rate}
 \meas\{\ell_{\boldsymbol t}(\boldsymbol x)>H\}
 \ll N^{-\Delta+o(1)},\qquad
 \Delta=\frac{\theta}{\log\phi}
 \left(H\cos\frac\pi{2000}-\frac13\log\rho\right).
\end{equation}
The phase grids have fixed size and do not affect the exponent of $N$.

\subsection{The large-deviation estimate for \texorpdfstring{$u=1,2,3$}{u=1,2,3}}
\begin{proof}[Proof of Lemma~\ref{lem:large}]
Appendix~\ref{app:cert} lists $32$ rational triples $(\boldsymbol t_i,\rho_i,H_i)$; the associated positive integer vectors have $377$ coordinates. They satisfy
\begin{equation}\label{eq:cert}
 \mathcal M(\boldsymbol t_i)v_i<\rho_i v_i,\qquad
 \frac{\theta}{\log\phi}\left(H_i\cos\frac\pi{2000}-\frac13\log\rho_i\right)
 >\frac{11667}{20000}.
\end{equation}
It remains to check that a large value of $T_e$ or $T_o$ forces at least one of these linear events.

Let $g_i(\boldsymbol x)=H_i-\ell_{\boldsymbol t_i}(\boldsymbol x)$.
For each parity, the certificate partitions $[0,1]^4$ into $17$ closed boxes whose endpoints are dyadic rationals, namely fractions with
power-of-two denominators. Their interiors are disjoint; shared boundary points are allowed. The checker reconstructs the bisections and verifies the entire coverage.
One box is handled by coordinate monotonicity: the upper corner already has $T_\nu\le\lambda$.
On each of the other $16$ boxes $\prod_{j=1}^4[l_j,u_j]$, an exact polynomial identity is supplied:
\begin{equation}\label{eq:poly}
 \lambda_0-T_\nu(\boldsymbol x)
 =\sum_\gamma q_\gamma\prod_j f_j(\boldsymbol x)^{\gamma_j}+R(\boldsymbol x),
 \qquad\lambda_0=0.341919.
\end{equation}
Here $q_\gamma\ge0$ are rational, the products have degree at most three, and each $f_j$ is a selected $g_i$ or a box-boundary factor $x_j-l_j$ or $u_j-x_j$. Here $[\boldsymbol x^\alpha]R$ denotes the coefficient of the monomial $\boldsymbol x^\alpha$ in $R$, and $\alpha$ runs over the exponent vectors of its monomials. Exact coefficient comparison verifies
\[
 \sum_\alpha|[\boldsymbol x^\alpha]R|<10^{-7}.
\]
If all $g_i\ge0$ at a point in the box, every product in~\eqref{eq:poly} is nonnegative. Also $|R(\boldsymbol x)|<10^{-7}$ on $[0,1]^4$. Therefore
\[
 T_\nu(\boldsymbol x)\le0.341919+10^{-7}<0.34192=\lambda.
\]
It follows that
\[
 \{T_e>\lambda\}\cup\{T_o>\lambda\}
 \subseteq\bigcup_{i=0}^{31}\{\ell_{\boldsymbol t_i}>H_i\}.
\]
Equations~\eqref{eq:rate} and~\eqref{eq:cert} finish the proof. The certificate covers the continuous amplitude cube; it is not a test on a finite grid of amplitudes.
\end{proof}

\subsection{Finite estimates for \texorpdfstring{$u=4$}{u=4}}
\begin{proof}[Proof of Lemma~\ref{lem:fourdat}]
For the major-arc constant, take $\boldsymbol k=(1,3,0,1)$ for even targets and $\boldsymbol k=(0,4,0,1)$ for odd targets in~\eqref{eq:Mk}. The minima are taken over even and odd residue classes, respectively. The exact rational sums give
\[
 \begin{aligned}
 M(1,3,0,1)&=1.9940595930033\ldots>1.99405,\\
 M(0,4,0,1)&=1.9939002468659648\ldots>1.9939.
 \end{aligned}
\]
In each case, all $122$ cyclic convolutions and their masses are recomputed by the checker. Section~\ref{sec:local} gives~\eqref{eq:fourmaj}. Taking the minimum over the wrong parity would give zero and would not prove this estimate.

For the large-deviation estimates, the parameters are
\begin{center}\small
\begin{tabular}{crrrrrr}
\toprule
Parity &$t_a$&$t_b$&$t_c$&$t_d$&$\rho$&$H$\\\midrule
Even&.544865&.283067&0&.311955&1.504612&.40684\\
Odd&0&.554889&0&.314443&1.504990&.40767\\\bottomrule
\end{tabular}
\end{center}
Use the same reduced state set $\mathcal S_+$ and period matrix as before. Let $Q_0=10^{12}$. Replace each of its six Fourier coefficient factors by an upward-rounded integer divided by $Q_0$, and call the resulting integer period matrix $\widetilde{\mathcal M}_\nu$.
The supplied positive integer vectors satisfy, coordinatewise,
\[
 \begin{aligned}
 376153Q_0^6v_e-250000\widetilde{\mathcal M}_ev_e&>0,\\
 150499Q_0^6v_o-100000\widetilde{\mathcal M}_ov_o&>0.
 \end{aligned}
\]
These are matrix--column-vector comparisons, not comparisons with the transpose.
They imply $\mathcal M(\boldsymbol t_\nu)v_\nu<\rho_\nu v_\nu$.

On $\mathcal E_{4,e}$, the arithmetic--geometric mean inequality gives
\[
 \frac{t_a|a|+2t_b|b|+2t_d|d|}{3}
 >(4\lambda_4t_at_bt_d)^{1/3}\ge H_e.
\]
On $\mathcal E_{4,o}$, apply it to $t_b|b|,t_b|b|,2t_d|d|$ to obtain
\[
 \frac{2t_b|b|+2t_d|d|}{3}
 >(2\lambda_4t_b^2t_d)^{1/3}\ge H_o.
\]
The last comparisons are the exact rational inequalities
$H_e^3\le4\lambda_4t_at_bt_d$ and $H_o^3\le2\lambda_4t_b^2t_d$ for their respective parameter rows.

The rate comparisons for $u=4$ can be made with four logarithm terms.
For $0<z<1$,
\[
 0<\log\frac{1+z}{1-z}
     -2\sum_{j=0}^{3}\frac{z^{2j+1}}{2j+1}
 \le\frac{2z^9}{9(1-z^2)}.
\]
This follows by replacing all remaining denominators by $9$ and summing a
geometric series. The following entries are upper bounds for the four terms
and their remainder, in units of $10^{-9}$:
\[
\begin{array}{c|rrr}
 &z_+=0.236068&z_+=0.2014732&z_+=0.2015937\\\hline
 2z_+&472136000&402946400&403187400\\
 2z_+^3/3&8770415&5452060&5461848\\
 2z_+^5/5&293256&132785&133182\\
 2z_+^7/7&11674&3850&3867\\
 2z_+^9/[9(1-z_+^2)]&536&127&128\\\hline
 \text{sum}&481211881&408535222&408786425
\end{array}
\]
All entries require only integer cross-multiplication. Since
$(2.236068)^2>5$, the first column bounds $\log\phi$; the other two
bound $\log(1.504612)$ and $\log(1.504990)$.
The needed comparisons of $z$ are
\[
 \frac{309017}{1309017}<0.236068,\quad
 \frac{126153}{626153}<0.2014732,\quad
 \frac{50499}{250499}<0.2015937.
\]
Thus we may use $\log\phi<0.481212$, $\log\rho_e<0.408536$ and
$\log\rho_o<0.408787$. With $\cos(\pi/2000)\ge1-(22/14000)^2/2$, the
rates are bounded below, respectively, by
\[
 \frac{568103551249439}{1010545200000000}>0.56217,
 \qquad
 \frac{15950761832556421}{28295265600000000}>0.5637.
\]
For the simpler thresholds $H_e=0.40684$, $H_o=0.40767$, the AM--GM checks are
\[
\begin{aligned}
 H_e^3&=0.067339662733504
       <0.06734014361942406819=4\lambda_4t_at_bt_d,\\
 H_o^3&=0.067752645897663
       <0.0677529329857700976594=2\lambda_4t_b^2t_d.
\end{aligned}
\]
These elementary comparisons replace numerical evaluations of logarithms,
cosines and cube roots; only the reduced matrix inequalities remain as
computer-assisted inputs to these two rate estimates.

Equation~\eqref{eq:rate} proves~\eqref{eq:fourmeas}.
\end{proof}

\subsection{Positive smoothing}
\begin{proof}[Proof of Lemma~\ref{lem:smooth}]
For sufficiently large $N$, we have $N^{0.99}<\tau N$.
Thus $R=S-S_w$ has nonnegative coefficients supported on two
endpoint intervals of total length $O(\tau N)$.
By the prime number theorem, the diagonal contribution is
$O(\tau N\log N/K)+o(N)=O(\tau N)+o(N)$.
Every nonzero recurrence shift is even and is
$O(N^\theta)=o(\tau N)$, so no prime pair joins the two endpoint
intervals.

We use the following uniform upper-bound sieve for prime pairs;
see \cite{IK} for the standard upper-bound sieve.
For any translated interval $I$ of length $X$, and any nonzero
even $h$ with $|h|=o(X)$,
\begin{equation}\label{eq:058}
 \#\{p,p+h\in I\}\ll\frac{X\mathfrak S(h)}{(\log X)^2}.
\end{equation}
This is the standard upper-bound sieve for two linear forms
\cite[Chapter 6]{IK}, applied to $n(n+h)$ on $I$. We only check the
uniformity needed here. The number of roots modulo a prime $p$ is one
if $p\mid h$ and two otherwise; at $p=2$ it is one. For squarefree $d$,
the remainder in the interval count is bounded by the product of these
root counts, independently of the position of $I$. The usual sieve with
$z=X^{1/10}$ and level $X^{1/2}$ therefore has error
$O(X^{1/2}(\log X)^C)$ and main term
$O(X\mathfrak S(h)/(\log X)^2)$. The cases in which one of the primes is
at most $z$ add $O(z)$. This proves the required uniformity in~\eqref{eq:058}.

Apply the estimate separately to the two endpoint intervals, taking
$X=2\tau N$. Since $\tau$ is fixed, $\log N/\log X\to1$.
Insert the logarithmic prime weights and sum using
\eqref{eq:coll}; this gives
\[
 \|Rb\|_2^2\ll \tau N+o(N).
\]
All implied constants use only the support length and $0\le w\le1$;
no derivative bounds for $w$ are involved.
Together with $\|Sb\|_2=O(\sqrt N)$, this bounds the change in
the ordinary energy by
$\|Rb\|_2(2\|Sb\|_2+\|Rb\|_2)$.
Since $|W_{4,\nu}|\le|b|^2$, the change in the major-arc integral is bounded
by the same expression. Both changes are therefore
$O(\sqrt\tau)N+o(N)$.
\end{proof}

Choose $\tau$ so that the losses preserve both positive margins in~\eqref{eq:margins}, and then fix $w$.
The constants in the remaining prime estimates may depend on this fixed weight; no uniformity as $\tau\to0$ is required.

\subsection{Parameters for the prime estimates}
\label{an:section}

In the following prime estimates, a Dirichlet polynomial means a finite sum $\sum_n a_n\chi(n)n^{-it}$; its length is the scale of the integers in its support. Fix the previously chosen $\tau>0$ and smooth
function $w$, with $0\le w\le1$ and
$\operatorname{supp}w\subset[\tau,1]$. Put
\begin{align*}
 V_w(\alpha)&=\sum_n\Lambda(n)w(n/N)e(n\alpha),\\
 S_w(\alpha)&=\sum_p(\log p)w(p/N)e(p\alpha).
\end{align*}
The following estimates do not depend on a recurrence weight. They prove
Lemma~\ref{lem:primein}.

For each $\alpha$, choose a reduced Dirichlet approximation
\[
 \alpha=a_0/q+\beta,\qquad q\le\sqrt N,\qquad
 |\beta|\le(q\sqrt N)^{-1},\qquad (a_0,q)=1,
\]
using the fixed ordering introduced in Section~\ref{sec:defs}. The resulting
sets of phases are measurable. Define
\[
 B_N=\max(q,qN|\beta|).
\]
On $\mathfrak m=\mathbb T\setminus\mathfrak M(P_+)$ we necessarily have
$B_N>P_+\ge N^{0.439}$, since otherwise this approximation places
$\alpha$ on a major arc. We call $B_N\le N^{0.46}$ the lower range
and $B_N>N^{0.46}$ the upper range of the minor arcs.
Denominator ranges always have upper endpoint
$Q$, so that $Q/2<q\le Q$, with $q=1$ treated separately. Similarly,
$T$ denotes the upper endpoint of a dyadic interval, that is, an interval
of the form $(T/2,T]$, for $\max(1,N|\beta|)$; the range
$T=1$ includes $N|\beta|\le1$. Henceforth $q\asymp Q$ refers to this
convention, and in particular $r\le q\le Q$.
There are $O((\log N)^2)$ ranges, and
\[
 QT\ll\sqrt N,\qquad Q^2T\ll N.
\]
The fixed endpoint factors give
$\log(QT)/\log N\le1/2+O(1/\log N)$. We first write the exponent inequalities
without this perturbation and retain the resulting $o(1)$ in all
applications. It is separate from the fixed loss $\epsilon=2\eta$ in
the small-time range and does not affect any fixed positive margin.
Integrals of nonnegative functions are bounded by the sum over all
admissible approximation intervals, whether or not those intervals
overlap.

\subsection{A large sieve with Gauss weights}
We average over both characters and a real parameter. A large sieve with
these two averages is usually called a hybrid large sieve.

The Gauss sum is the finite exponential sum
$\tau_q(\chi)=\sum_{a_0\bmod q}\chi(a_0)e(a_0/q)$, extending characters by zero
on residue classes not coprime to $q$. A superscript asterisk on a character sum restricts it to primitive characters. We write $\tau(n)$ for the number of positive divisors of $n$; the unadorned $\tau$ remains the fixed support parameter.

\begin{lemma}\label{an:sieve}
Suppose that $1\le Q,T\le N$ and that the coefficients $A_n$ are
independent of the modulus and the character and are supported on
$1\le n\le CN$, where $C$ is fixed. Then
\begin{equation}\label{eq:061}
\begin{split}
 &\sum_{q\le Q}\frac1q\sum_{\chi\bmod q}|\tau_q(\chi)|^2
 \int_{-T}^T\left|\sum_n A_n\chi(n)n^{-it}\right|^2dt\\
 &\hspace{2em}\ll N^{o(1)}(N+Q^2T)\sum_n|A_n|^2.
\end{split}
\end{equation}
The same bound holds when the time integral is replaced by a sum over
a common lattice $t=j+\delta$, $0\le\delta<1$, uniformly in $\delta$.
\end{lemma}

\begin{proof}
We use Gallagher's hybrid large sieve~\cite{Gallagher} in the standard
primitive-character form
\[
 \sum_{r\le R}\frac r{\varphi(r)}\sum_{\chi\bmod r}^{*}
 \int_{-T}^{T}\left|\sum_{n\le L}a_n\chi(n)n^{-it}\right|^2\dd t
 \ll (L+R^2T)\sum_{n\le L}|a_n|^2 \qquad(T\ge1).
\]
It also follows from the additive version in~\cite[Lemma 1]{BB}, with its
parameter $f=1$, by the primitive Gauss identity and Bessel's inequality.
Here $L$ is the length of the coefficient sequence and the star restricts
the sum to primitive characters. We explain only the deduction
with the Gauss weights required in~\eqref{eq:061}.

If a character modulo $q=rk$ is induced by a primitive character
$\chi^*$ modulo $r$, then
\[
 \tau_q(\chi)=\mu(k)\chi^*(k)\tau_r(\chi^*).
\]
A nonzero term requires $k$ to be squarefree and $(k,r)=1$. In that
case $|\tau_q(\chi)|^2=r$, so its weight in
\eqref{eq:061} is $1/k$. Expand the coprimality indicator as
$\mathbf1_{(n,k)=1}=\sum_{d\mid(n,k)}\mu(d)$ and apply Cauchy's
inequality to the sum over $d\mid k$, losing only
$\tau(k)=N^{o(1)}$. For fixed $k,d$, the primitive hybrid large sieve
gives
\[
 \left(\frac Nd+\frac{Q^2T}{k^2}\right)E_d,
 \qquad E_d=\sum_m|A_{dm}|^2.
\]
Writing $k=d\ell$ and summing with weight $1/k$, the two coefficients
are at most $N\log Q/d^2$ and $O(Q^2T/d^3)$, respectively. The bound
$E_d\le\sum_n|A_n|^2$ and the convergence of both series in $d$ prove
the continuous estimate.

For a Dirichlet polynomial $D(t)=\sum_n A_n\chi(n)n^{-it}$ and sampling points with $|t_j-t_k|\ge1$ for $j\ne k$, use the local Sobolev inequality
\[
 |D(t_j)|^2\ll\int_{t_j-1/2}^{t_j+1/2}
 \bigl(|D(t)|^2+|D(t)D'(t)|\bigr)dt.
\]
The intervals are disjoint. Cauchy's inequality and the continuous
estimate apply to both $D$ and $D'$. Differentiation multiplies the
coefficients only by $\log n$, at a logarithmic cost. A common time
translation multiplies every $A_n$ by the same unit-modulus factor
$n^{-i\delta}$ and preserves independence from the modulus. The bound
is therefore uniform in $\delta$.
\end{proof}

\subsection{The Heath--Brown identity and Mellin inversion}

The fifth-order Heath--Brown identity~\cite{HB} gives,
for every $n\le N$,
\begin{equation}\label{an:HB}
\begin{split}
 \Lambda(n)
 &=\sum_{j=1}^{5}(-1)^{j-1}\binom5j
 \sum_{\substack{m_1\cdots m_{2j}=n\\m_{j+1},\ldots,m_{2j}\le N^{1/5}}}
 (\log m_1)\mu(m_{j+1})\cdots\mu(m_{2j}).
\end{split}
\end{equation}
Decompose each variable into dyadic intervals, meaning intervals whose endpoints differ by a factor of two. For the unrestricted
variables, use complete intervals covering $[1,N]$ with upper endpoints
at most $2N$, treating $m_i=1$ separately if necessary. The last
interval for each M\"obius variable is truncated at $N^{1/5}$. This gives
$O((\log N)^{10})$ complete convolutions $h_\ell$.
Retain only boxes whose product support meets $[\tau N,N]$. Their length
products lie in $[\tau N/2^{10},N]$, and their complete supports lie in
$[c_\tau N,C_\tau N]$. Each $h_\ell(n)$ is $N^{o(1)}$, is independent of
$q,a_0,\beta,\chi$, and satisfies
\begin{equation}\label{eq:063}
 \sum_n|h_\ell(n)|^2\le N^{1+o(1)}.
\end{equation}
The complete convolution coefficients for $n>N$ are retained and are
subsequently removed by $w(n/N)$. The restriction on the total product
thus remains in the external weight rather than in the individual
factors.

In the large-time ranges, take the merging parameter $B=QT$. In the
small-time ranges, first set
\begin{equation}\label{an:eta}
 \eta=10^{-6},\qquad U=N^{2\eta},\qquad B=QU,
\end{equation}
and then construct the decomposition for that range. On the upper
range of the minor arcs, both choices satisfy
$N^{4/9}\ll B\le N^{1/2+2\eta}\ll \tau N$.
Following the construction in \cite[Lemmas 3.2--3.3]{MPR}, whenever
the product of two factor lengths is less than $2^{-20}B^{1/2}$,
replace the factors by their complete convolution. There are at most
nine merges. The upper endpoint of each newly merged support is at
most $2^{-10}B^{1/2}$, and at termination every pair of lengths has
product at least $2^{-20}B^{1/2}$. A factor of length greater than
$B^{1/2}$ must be an unmerged original factor with coefficients $1$
or $\log n$, since the original M\"obius factors have lengths at most
$N^{1/5}$. Merging does not change $h_\ell$.

For a complete box, write
\begin{align*}
 \mathcal H_q(t,\chi)&=\sum_nh_\ell(n)\chi(n)n^{-it},\\
 \mathcal K_\beta(t)&=\int_0^\infty
 w(y)e(\beta Ny)y^{it}\frac{dy}{y}.
\end{align*}
Fourier inversion in $\log y$, equivalently Mellin inversion in $y$, gives
the exact identity
\begin{equation*}
 w(n/N)e(\beta n)=\frac1{2\pi}\int_{\mathbb R}
 \mathcal K_\beta(t)N^{it}n^{-it}dt.
\end{equation*}
Character orthogonality then expresses the part of the box coprime to
$q$ as
\begin{equation*}
 \frac1{2\pi\varphi(q)}\sum_{\chi\bmod q}
 \tau_q(\overline\chi)\chi(a_0)
 \int_{\mathbb R}\mathcal H_q(t,\chi)N^{it}\mathcal K_\beta(t)dt.
\end{equation*}
There is a single Mellin variable. The merging scheme may vary from
one range to another: identity~\eqref{an:HB} is applied coefficientwise
within each range, and the nonnegative integrals are summed afterwards.

\subsection{A fourth integral and additive energy}
Fix a complete convolution, a denominator range, and a common set of sampling
points $t=j+\delta$, $j\in\Z$, $0\le\delta<1$. For each $q$ retain only
characters with nonzero Gauss sum, and group them by primitive conductor
$r\mid q$. Partition these pairs $(t,\chi)$ so that, on each set, the
absolute values of every factor and grouped product lie in fixed intervals
$(A/2,A]$, with the endpoint chosen separately for each factor and product. Also partition $j$ by its residue class modulo the spacing integer
specified below. Later restrictions and selections are made by fixed rules,
so the resulting sets remain disjoint. These partitions do not change the
coefficients of the complete convolution.

For any such finite set $\mathcal W$, define its additive energy by
\[
 \mathcal E(\mathcal W)=\#\left\{((t_i,\chi_i))_{i=1}^4\in\mathcal W^4:
 \begin{array}{l}|t_1+t_2-t_3-t_4|\le1,\\
 \chi_1\chi_2=\chi_3\chi_4\end{array}\right\}.
\]
It counts ordered quadruples with nearly equal time sums and equal character
products. Chen--Gupta--Li~\cite[Lemma 9.2]{CGL} use this quantity to bound a
character fourth moment. The next lemma applies the same relation after the
integral expansion of the prime sum above. We keep the conductor and measure factors
explicit; the coefficients can now be complex.

\begin{lemma}\label{lem:cell}
In a range with $T\ge N^\eta$, fix a conductor $r\mid q$, a common
set of sampling points, and one nonempty part $\mathcal W$ of the preceding partition. Let $K=|\mathcal W|$, and
let $A$ be a dyadic upper bound for the modulus of the complete product
on the subset, comparable to that absolute value by a fixed factor. The contribution of this subset to the fourth integral is at most
\begin{equation}\label{eq:066}
 N^{o(1)}\frac{r^2}{q^3TN}A^4\mathcal E(\mathcal W).
\end{equation}
Moreover, disjoint subsets satisfy
\begin{equation}\label{eq:067}
 \sum_{q\asymp Q}\sum_{r,\,\mathrm{subsets}}\frac rq KA^2
 \ll N^{2+o(1)}.
\end{equation}
The small fixed losses from normalization and the time partition are accounted
for separately and chosen together at the end of the section.
\end{lemma}

\begin{proof}
Write $\beta=\pm Tv/N$. The current dyadic range has
$1/2\le v\le1$, which lies in the interior of the fixed interval
$[1/4,3]$. The phase
$2\pi\beta Ny+t\log y$ has stationary point
$y_0=\mp t/(2\pi Tv)$ and stationary value
$t\log(|t|/(2\pi T))-t-t\log v$.
Put $\Phi_T(t)=t\log N+t\log(|t|/(2\pi T))-t$, the part of this phase independent of $v$, including the factor $N^{it}$. For any fixed expansion order $L$, the stationary-phase expansion gives
\begin{align*}
 N^{it}\mathcal K_\beta(t)
 &=T^{-1/2}e^{i\Phi_T(t)}v^{-it}
 \sum_{\ell<L}T^{-\ell}f_\ell(t/T,v)\\
 &\quad+O_{L,\tau,w}(T^{-L}).
\end{align*}
The nonzero main amplitudes are supported where $c_\tau\le |t|/T\le C_\tau$,
for fixed positive constants $c_\tau,C_\tau$. All their derivatives are
uniformly bounded. The positive constant $C$ in the Fourier expansion below
is chosen large enough to contain this rectangle.
Since $w$ vanishes smoothly at its support boundary, the amplitudes
extend smoothly as the stationary point leaves the support.
Repeated integration by parts treats the nonstationary portion and
the tail $|t|>C_\tau T$. Each $f_\ell$ has an absolutely summable
two-dimensional Fourier expansion
\begin{align*}
 f_\ell(t/T,v)&=\sum_{m,n}c_{m,n,\ell}
 e^{2\pi i mt/(CT)}e^{2\pi i nv/C},\\
 \sum_{m,n}|c_{m,n,\ell}|&\ll_{\tau,w,L}1.
\end{align*}
The first exponential multiplies the coefficient at the same time
$t$ and does not translate the time argument of $\mathcal H_q$.
The second lies outside the character--time sum and has modulus one.
Minkowski's inequality therefore reduces the argument to the kernel
$v^{-it}$.

Use the identity
$\int g(t)dt=\int_0^1\sum_jg(j+\delta)d\delta$.
The fourth norm over all $q,a_0,v$ uses the same $\delta$. Minkowski's
inequality on this interval of measure one causes no additional
length loss. Partition the integers $j$ by their residue classes modulo
$\max(3,\lceil(C_\tau rT)^\zeta\rceil)$, where $\zeta>0$ may be
arbitrarily small. Times in the same residue class with the same character have the
separation required by Proposition~\ref{prop:CGL}, at a cost of $N^{O(\zeta)}$.
The common $\delta$ cancels from each fourfold time difference. A
subset with a fixed residue class of $j$ may thus be viewed as a subset of the product of
the character group modulo $r$ and the integer lattice.
Let $F(\chi,j)$ count ordered pairs of points in $\mathcal W$ whose
characters have product $\chi$ and whose integer coordinates have sum $j$.
Thus $F=\mathbf1_{\mathcal W}*\mathbf1_{\mathcal W}$ is convolution on the
character group and the integer lattice. With counting measure on this discrete group, the spacing of at least three gives
\begin{align*}
 \|F\|_2^2&=\mathcal E(\mathcal W),\\
 \left|\sum_{\chi,j}F(\chi,j+h)\overline{F(\chi,j)}\right|
 &\le\mathcal E(\mathcal W)\qquad(h\in\mathbb Z).
\end{align*}
Take a fixed smooth nonnegative function in $v$, supported on
$[1/4,3]$ and equal to one on $[1/2,1]$. It majorizes the indicator
of the current range, and its transform in $\log v$ decays rapidly. Expanding the
fourth power and applying character orthogonality in $a_0$ retains only
equal character products. The surjectivity of reduction from the
unit group modulo $q$ to that modulo $r$ makes this equality condition
for induced characters identical to the one modulo $r$.
For coefficients $c_{\chi,t}$ with $|c_{\chi,t}|\le A$, the terms in the
fourfold expansion at integer time difference $h$ have total absolute value
at most $A^4\sum_{\chi,j}F(\chi,j+h)F(\chi,j)$. The preceding inequality
and the rapid decay of the transform of the smooth weight therefore give
\begin{align*}
 &\int\sum_{(a_0,q)=1}
 \left|\sum_{\mathcal W}c_{\chi,t}\chi(a_0)v^{-it}\right|^4dv\\
 &\hspace{2em}\ll\varphi(q)A^4\mathcal E(\mathcal W),
 \qquad |c_{\chi,t}|\le A.
\end{align*}
Multiply
by $(\sqrt r/\varphi(q))^4T^{-2}$, use
$d\beta=(T/N)dv$ and $q/\varphi(q)=N^{o(1)}$, and obtain
\eqref{eq:066}.

For each modulus, the conductor sum is bounded by
$|\sum_{r\mid q}Z_r|^4\le\tau(q)^3\sum_{r\mid q}|Z_r|^4$.
The subsets are disjoint and $A\asymp|\mathcal H_q|$, so
\begin{align*}
 \sum_{r,\,\mathrm{subsets}}\frac rq KA^2
 &\ll\frac1q\sum_{\chi\bmod q}|\tau_q(\chi)|^2
 \sum_{|j|\le C_\tau T}|\mathcal H_q(j+\delta,\chi)|^2.
\end{align*}
Lemma~\ref{an:sieve}, \eqref{eq:063}, and $Q^2T\ll N$
prove \eqref{eq:067}.

To bound the errors, note that $|\mathcal H_q|\le N^{1+\nu}$ for
every fixed $\nu>0$ and all sufficiently large $N$, and that
$\varphi(q)^{-1}\sum_\chi|\tau_q(\chi)|\le\sqrt q$.
The time integral of the absolute value of the stationary and
nonstationary remainders is $O(T^{1-L})$. Choose a fixed sufficiently
large integer $L$, depending only on $\eta,\tau,w,\nu$, such that
$N^{5/4+\nu}T^{1-L}\le N^{-100}$. These errors are negligible in all
ranges and fourth norms.
\end{proof}

\subsection{Poles, remainders, and sampling for long factors}

For an original factor with coefficients $f(n)=1$ or $\log n$ on
$(P,2P]$, define
\begin{align*}
 S_{q,\chi^*}(t)&=\sum_{\substack{P<n\le2P\\(n,q/r)=1}}
 f(n)\chi^*(n)n^{-it},\\
 P_{q,1}(t)&=\frac{\varphi(q)}q\int_P^{2P}f(x)x^{-it}dx.
\end{align*}
Here $\chi^*$ is primitive modulo $r$. Set $P_{q,\chi^*}=0$ when
$r>1$, and write $R_{q,\chi^*}=S_{q,\chi^*}-P_{q,\chi^*}$ in all
cases. A pole occurs only for the primitive principal character
$r=1$. When $f=\log n$, the negative derivative of the Dirichlet
series with the coprimality restriction has the single negative-power
Laurent term $(\varphi(q)/q)(s-1)^{-2}$ at $s=1$; its residue integral
is precisely the pole term above. Expanding the coprimality indicator
over $d\mid q/r$ and combining the $\log d$ terms gives the same
function.

We use the fourth-moment estimate for the remainder proved inside
\cite[Lemma 3.8, especially (3.3)--(3.4)]{MPR}. In the form needed here,
let $R_d(t)$ be the remainder for the primitive-character polynomial on
$(P/d,2P/d]$, after subtracting its pole term when $r=1$. The cited argument,
using the standard fourth moments for $L$ and $L'$ on the critical line
\cite[Chapter 10]{Montgomery}, gives
\[
 \sum_{\chi^*\bmod r}^{*}\int_{-V}^{V}|R_d(t)|^4\dd t
 \ll rV(P/d)^2N^{o(1)}\qquad(V\ge1).
\]
Only intervals centred at zero are used. For $(P/d)^2\le rV$, the ordinary
mean-value estimate applied to the square of the polynomial gives the same
bound. The contour argument in the cited proof covers the other range.

For the coprimality restriction in our factor, expand
$\ind_{(n,q/r)=1}=\sum_{d\mid(n,q/r)}\mu(d)$ and put $n=dm$.
For $f(n)=1$, the new coefficient outside the polynomial is
$\mu(d)\chi^*(d)d^{-it}$, of modulus at most one. For $f(n)=\log n$,
the identity $\log(dm)=\log d+\log m$ gives a linear combination of the
same two types of polynomial; the additional factor $\log d$ has only a
logarithmic cost. If $P/d<1$, a nonempty interval contains only a bounded
number of integers. An empty interval has pole contribution
$O((P/d)\log^C N/(1+|t|))$. These cases satisfy the same bound directly.

Combine all pole terms from the coprimality expansion before applying
Minkowski's inequality. Since
$\sum_{d\mid q/r}d^{-1/2}\le\tau(q/r)=N^{o(1)}$, for every $V\ge1$
we obtain
\begin{equation}\label{eq:068}
\begin{split}
 &\sum_{\chi^*\bmod r}^{*}\int_{-V}^{V}|R_{q,\chi^*}(t)|^4dt\\
 &\hspace{2em}\ll rVP^2
 \bigl(\log(2NPr(V+1))\bigr)^C N^{o(1)}.
\end{split}
\end{equation}

The full remainder $R=S-P$ is the Fourier transform of a finite
complex measure with frequency support in $[\log P,\log(2P)]$.
Choose a smooth function $K$ such that it and all its derivatives decrease
faster than every inverse power, and whose Fourier transform equals one on
$[-1,1]$, and set $D=C\log(2N)$, with $C$ sufficiently large, and
$K_D(t)=DK(Dt)$. Then $R=K_D*R$, and H\"older's inequality gives,
for every fixed $A_0>0$,
\[
 |R(t_j)|^4\ll_{A_0}D\int_{\mathbb R}
 (1+D|t-t_j|)^{-A_0}|R(t)|^4dt.
\]
Sum over any $1$-separated set of times in
$[-\mathcal T,\mathcal T]$. The sum of the kernels is $O(D)$ on
$[-\mathcal T-1,\mathcal T+1]$ and is
\[
 O_{A_0}\bigl(D(1+\operatorname{dist}
 (t,[-\mathcal T,\mathcal T]))^{1-A_0}\bigr)
\]
outside that interval. Take $A_0>C+10$, split the exterior into
dyadic distance ranges, and apply \eqref{eq:068}.
This gives
\begin{equation}\label{eq:069}
 \sum_{\chi^*\bmod r}^{*}\sum_{|t_j|\le\mathcal T}
 |R_{q,\chi^*}(t_j)|^4\ll r\mathcal T P^2N^{o(1)}.
\end{equation}
This sampling estimate follows directly from the continuous remainder
bound; it requires neither a time derivative of the remainder nor a
fourth-moment bound over translated critical-line intervals.

We separate the points $(\chi^*,t)$ where at least one original factor
of length greater than $B^{1/2}$ satisfies $|S_j|<2|P_j|$. Here $S_j$ is
that factor and $P_j$ is its pole contribution as defined above. We refer
to these points as the pole part. Since $P_j=0$ unless $r=1$ and
$|P_j(t)|\ll N^{a_j+o(1)}/(1+|t|)$, the complete product on this set
is $O(N^{1+o(1)}/(1+|t|))$. In the large-time range, the Mellin
integral involves only $|t|\asymp_\tau T$, and the principal Gauss sum
has modulus at most one. Thus the pole contribution at each fraction
is at most $N^{1+o(1)}/(q\sqrt T)$, and its total fourth integral is
\begin{equation}\label{eq:070}
 \ll\frac{N^{3+o(1)}}{Q^2T}
 =\frac{N^{3+o(1)}}{Q(QT)}\le N^{2.54+o(1)}.
\end{equation}
For small times, the direct kernel bound and
$\int_{-U}^U(1+|t|)^{-1}dt\ll\log N$ give
$N^{3+o(1)}T/Q^2\le N^{2.08+3\eta+o(1)}$.
On the remaining set, each original long factor satisfies
$|S_j|\le2|R_j|$. Estimate~\eqref{eq:069} therefore
controls the values of the original polynomials on that set.

\Needspace{10\baselineskip}
\subsection{The cited energy estimate and factor moments}

\begin{proposition}[Chen--Gupta--Li {\cite[Proposition 10.2]{CGL}}]\label{prop:CGL}
Fix $\zeta>0$ and a real number $\sigma$. Let $D(t,\chi)=\sum_{P<n\le2P}c_n\chi(n)n^{it}$ with $|c_n|\le1$.
Suppose that $\chi$ is primitive modulo $r$, that $|t|\le\mathcal T$
on $\mathcal W$, that distinct times with the same character are
separated by at least $(r\mathcal T)^\zeta$, and that
\begin{gather*}
 |D(t,\chi)|\ge P^\sigma/6\qquad((t,\chi)\in\mathcal W),\\
 (r\mathcal T)^{2/3}/2\le P\le r\mathcal T.
\end{gather*}
If at least one of the conditions
\[
 K\ge(r\mathcal T)^{2/3},\qquad
 \mathcal E(\mathcal W)^{1/8}P
 \ge(r\mathcal T)^{1/2}K^{5/8}
\]
holds, then
\begin{equation*}
\begin{split}
 \mathcal E(\mathcal W)\ll_{\zeta,\varepsilon}
 (r\mathcal T)^\varepsilon\bigl(&KP^{4-4\sigma}
 +K^{21/8}(r\mathcal T)^{1/4}P^{1-2\sigma}\\
 &+K^3P^{1-2\sigma}\bigr).
\end{split}
\end{equation*}
for every fixed $\varepsilon>0$, as $r\mathcal T\to\infty$.
Here $K=|\mathcal W|$ and $\mathcal E(\mathcal W)$ is the energy defined above.
\end{proposition}
The free $\varepsilon$ in its conclusion is distinct from the fixed separation
exponent $\zeta$. If the additional disjunction fails, its negation itself gives
the separate complementary bound
\begin{equation*}
 \mathcal E(\mathcal W)<(r\mathcal T)^4K^5/P^8.
\end{equation*}
The arbitrarily small losses from coefficient normalization and the time partition in these
estimates are accounted for at the end of the section.

For large times, the actual time bound is a fixed multiple of $T$; for small times it is $U$.
Increasing these bounds by a further fixed factor is harmless when applying the quoted estimate.
In the exponent calculations below we write $\mathcal T=T$ or $U$ and absorb these fixed factors into the displayed $o(1)$ terms.
Thus the time and length inequalities are enforced exactly before their exponents are used.
Write
\begin{gather*}
 Q\mathcal T=N^{b_0},\qquad r\mathcal T=N^{b_1},\\
 K=N^k,\qquad A=N^\upsilon,\qquad
 \mathcal E(\mathcal W)=N^e,\\
 \xi=\upsilon+(e-k)/2,\qquad
 \xi_{\rm eff}=\xi-(b_0-b_1)/2.
\end{gather*}
The exponents $b_0,b_1$ record the products $Q\mathcal T,r\mathcal T$;
$k,\upsilon,e$ record the size of the finite set, the product bound, and its energy. In particular,
$N^\xi=A\sqrt{\mathcal E(\mathcal W)/K}$. The corrected exponent
$\xi_{\rm eff}$ includes the factor $(r/Q)^{1/2}$ from the conductor.
$r$ remains the conductor and $q$ the original modulus.
Empty subsets are discarded. A nonempty subset with a fixed residue class of $j$ has
$0\le k\le b_1+o(1)$ and $e\le3k$.
For an actual factor of length $N^a$ and absolute value at least $N^{u_a}$,
the ordinary sampled mean value estimate, whose continuous form is
\cite[Lemma 3.4]{MPR}, gives, for every fixed integer $j\ge1$,
\begin{equation}\label{eq:071}
 KN^{2ju_a}\le N^{\max(ja,b_1)+ja+o(1)}.
\end{equation}
This is the character mean value estimate applied to the $j$th power
of the polynomial, whose coefficients are bounded by a fixed power of $\log N$ times the
number of ways to factor an integer into a fixed number of positive factors. As in Lemma~\ref{an:sieve}, local Sobolev and the
logarithmic cost of differentiation give the sampled version.

Each subset records the dyadic absolute-value ranges of the actual original factors
and the actual grouped products, together with the choice of whether a point lies in the pole part
defined above. If a group $M_i$ is not supported on a single dyadic
interval, write it as a sum of a bounded number of dyadic pieces,
$M_i=\sum_hD_{i,h}$. At each point, select the first piece satisfying
$|D_{i,h}|\ge|M_i|/L_i$, where $L_i$ is the number of pieces, and
refine the subsets by this choice. Apply Proposition~\ref{prop:CGL} to the selected
polynomial, with threshold equal to the lower bound for the actual
$M_i$ divided by $L_i$. The mean value constraints continue to apply
to the actual $M_i$, and the large sieve continues to apply to the
complete $\mathcal H_q$. All constraints therefore hold on the same
set, and the nonnegative sum in \eqref{eq:067} remains bounded
by the original large sieve estimate.

Absorb $\mathbf1_{(n,q/r)=1}$ into the coefficients and apply Proposition~\ref{prop:CGL} to
the primitive character $\chi^*$. The coefficients count factorizations
into a fixed number of factors, with at most logarithmic weights, and are
therefore bounded by $N^\varepsilon$ for each fixed $\varepsilon>0$. Dividing them by $N^\rho$ makes their moduli at most one, where
$\rho>0$ is chosen arbitrarily small at the end. Reflecting time
changes $n^{-it}$ into the $n^{it}$ convention of Proposition~\ref{prop:CGL} and preserves
energy and separation. We retain separately the exponent losses due
to coefficient normalization, constants from the dyadic pieces, and
dyadic value ranges.
At points where any actual factor has modulus less than $N^{-10}$,
trivial bounds for the other factors, the character sum, and the
time interval give a total contribution at most $N^{-8+o(1)}$.
The remaining absolute values require only $O(\log N)$ dyadic levels. Since
there are at most ten factors, the number of refinements is a power
of $\log N$.

Divide the length exponents by their sum, which is $1+O_{\tau}(1/\log N)$.
The pairwise lower bound is then $b_0/2+o(1)\ge0.23+o(1)>2/9$.
Thus \cite[Lemma 3.2]{MPR} applies with its parameter equal to zero,
not with the small-time parameter $\eta$ in~\eqref{an:eta}.
For five or more factors it gives a partition
either into two groups of normalized length exponents $\gamma,1-\gamma$, with
$5/12\le\gamma\le1/2$, or into three groups satisfying
\begin{equation*}
\begin{gathered}
 \gamma_1\ge\gamma_2\ge2\gamma_3\ge1/3,\qquad
 \gamma_1\le5/12,\\
 \gamma_1+\gamma_2+\gamma_3=1.
\end{gathered}
\end{equation*}
In particular, $\gamma_1,\gamma_2\ge7/18$ and
$1/6\le\gamma_3\le1/5$. Fixed constants in the length products and
$o(1)$ perturbations of the boundaries in the combinatorial lemma
cause only arbitrarily small exponent losses.
For two groups, Proposition~\ref{prop:CGL} is used only when $b_1\ge\gamma$; for three
groups, it is used on the first two only when $b_1\ge5/12$.
If the length of the selected polynomial has an additional fixed factor, first multiply
the time parameter by a fixed constant so that $P\le r\mathcal T$
holds exactly, and then partition the sampling points using that enlarged parameter. The lower length
condition is preserved, since in the original or enlarged range,
respectively,
\[
 \frac{2b_1}3\le\frac13+\frac{2\epsilon}3+o(1),
 \qquad \epsilon=0\ \text{or}\ 2\eta,
\]
whereas every length used is at least $N^{7/18+o(1)}$. Fixed factors
cannot remove this positive exponent gap. Thus all the primitivity,
coefficient, length, time, separation, and value hypotheses hold on
the same subset.

\subsection{Estimates for the factor groups}

\begin{lemma}\label{an:budget}
Let $\varepsilon_*>0$ be fixed. The small-factor cutoff and coefficient
normalization can be chosen so that, for $0.46\le b_0\le1/2$,
\[
 \xi_{\rm eff}\le
 \max(1,1-\gamma/2+7b_0/16)+\varepsilon_*+o(1)
\]
in the two-group case. In every other case outside the pole part,
$\xi_{\rm eff}\le1+\varepsilon_*+o(1)$.
For $b_0\le1/2+\epsilon$, $\epsilon=2\eta$, replace the latter bound by
$1+\epsilon+\varepsilon_*+o(1)$.
In the two-group B case the enlarged-range bound is
\[
 \xi_{\rm eff}\le1-\gamma/2+7b_0/16+\epsilon/2+
 \varepsilon_*+o(1).
\]
The loss from summing time residue classes is separate from these estimates for individual subsets.
\end{lemma}

\begin{proof}
We first give the exponent calculations before the fixed auxiliary losses.
Every inequality in the case calculations below records its leading exponent;
the omitted fixed error is not asserted to be $o(1)$. Its total effect is
bounded by $\varepsilon_*$ at the end of this proof.
For length $P=N^p$ and actual group absolute value $N^z$, the three terms in Proposition~\ref{prop:CGL} have exponents
\begin{align*}
 \mathsf A&=k+4p-4z,\\
 \mathsf B&=21k/8+b_1/4+p-2z,\\
 \mathsf C&=3k+p-2z.
\end{align*}
We label the three terms A, B, and C in their displayed order.
When the extra condition fails, the bound is $e<4b_1+5k-8p$. Divide into cases according
to the largest of the three terms; the factor three contributes only
$o(1)$ to the exponents. We always have $\xi\le\upsilon+k$.

First consider two groups. If $b_1<\gamma$, both lengths exceed $r\mathcal T$,
and their second moments and Cauchy's inequality give
$KA\le N^{1+o(1)}$. Suppose that $b_1\ge\gamma$, and let $z_1,z_2$
be the absolute-value exponents of the groups. In the original range,
\[
 z_1\le(b_1+\gamma-k)/2,\qquad
 z_2\le1-\gamma-k/2.
\]
Case A gives $\xi\le z_2+2\gamma-z_1$. Averaging this with
$\xi\le z_1+z_2+k$ yields
$\xi\le z_2+\gamma+k/2\le1$.
Case B gives
\[
 \xi\le z_2+13k/16+b_1/8+\gamma/2
 \le1-\gamma/2+7b_1/16.
\]
In the complementary case,
\begin{align*}
 \xi&<z_1+z_2+2b_1+2k-4\gamma\\
 &\le1+7b_1/2-9\gamma/2\le7/8.
\end{align*}

Case C gives $\xi\le z_2+k+\gamma/2$ and requires additional factor
moments. For a factor whose length exponent satisfies $a\le b_1/2$
and exceeds a fixed small cutoff, take
$j=\lceil2b_1/(3a)\rceil$. If $j=2$, then $ja\le b_1$.
If $j\ge3$, then $a<b_1/3$, so
$ja<2b_1/3+a<b_1$. Hence $2b_1/3\le ja\le b_1$ and
$1/j\le3a/(2b_1)$. By \eqref{eq:071},
\[
 u_a\le a/2+3a(b_1-k)/(4b_1).
\]
There are at most ten factors below the cutoff. Trivial bounds for
them give an arbitrarily small total loss.
For an original long factor outside the pole part,
\eqref{eq:069} gives
\[
 u_a\le a/2+(b_1-k)/4
 \le a/2+3a(b_1-k)/(4b_1),
\]
because $a>b_0/2\ge b_1/2$.
For the remaining factors with $b_1/2<a\le b_0/2$, including merged
factors and unmerged short factors, the fourth moment gives
$u_a\le a-k/4$. Its excess over the preceding bound is
\[
 -a/4+(3a/(4b_1)-1/4)k
 \le(2a-b_1)/4\le(b_0-b_1)/4.
\]
There are at most two such factors in the group of length exponent
$1-\gamma\le7/12$, since $b_1\ge5/12$. Consequently,
\[
 z_2\le\frac{1-\gamma}2+
 \frac{3(1-\gamma)}{4b_1}(b_1-k)+\frac{b_0-b_1}2.
\]
Put $c=3(1-\gamma)/(4b_1)$. Subtracting the conductor term gives
$\xi_{\rm eff}\le1/2+cb_1+(1-c)k$.
If $c\le1$, the maximum occurs at $k=b_1$ and is at most
$1/2+b_1$. If $c>1$, it occurs at $k=0$ and is at most
$1/2+3(1-\gamma)/4\le15/16$.

In the enlarged range, the precise mean value bound for the second
group is
\begin{align*}
 z_2&\le\frac{\max(1-\gamma,b_1)+1-\gamma-k}{2}\\*
 &\le1-\gamma-k/2+\epsilon/2.
\end{align*}
This includes $1-\gamma<b_1$. Thus case A gives $1+\epsilon/2$,
case B gives $1-\gamma/2+7b_1/16+\epsilon/2$, and the complementary
case gives $7/8+4\epsilon$. The factor argument for case C is
unchanged and yields $\xi_{\rm eff}\le1+\epsilon$.
After subtracting the conductor term, the case B bound may be relaxed
to $1-\gamma/2+7b_0/16+\epsilon/2$.

Now consider three groups. If $b_1<\gamma_1$, combine the last two
groups and apply Cauchy's inequality to obtain again
$KA\le N^{1+o(1)}$. If $\gamma_1\le b_1<5/12$, the three ordinary
mean value estimates give
\begin{align*}
 \upsilon+k&\le1/2+5b_1/4-k/4,\\
 \xi_{\rm eff}&\le1/2+7b_1/4-b_0/2<1,
\end{align*}
where the last inequality uses $b_0\ge0.46$.
Suppose now that $b_1\ge5/12$, and write the absolute values of the groups as
$N^{z_1},N^{z_2},N^{z_3}$. On the same subset, we have
\begin{equation}\label{eq:073}
\begin{aligned}
 z_i&\le(b_1+\gamma_i-k)/2\qquad(i=1,2),\\
 z_3&\le(b_1+2\gamma_3-k)/4,\\
 z_1+z_2&\le b_1+(1-\gamma_3)/2-k,\\
 \upsilon&\le1/2+5b_1/4-5k/4.
\end{aligned}
\end{equation}
There is also the trivial bound $z_3\le\gamma_3+o(1)$.
If either of the first two groups fails the extra condition in Proposition~\ref{prop:CGL}, then
\begin{align*}
 \xi&<\upsilon+2b_1+2k-4\gamma_i\\
 &\le1/2+13b_1/4+3k/4-4\gamma_i\\
 &\le17/18+4\epsilon,
\end{align*}
which has a fixed positive margin.

Otherwise each group is in case A, B, or C. Taking the indicated
convex combinations of the two energy bounds and then using
\eqref{eq:073} covers all nine possibilities:
\[
\begin{array}{c|c|c}
 \text{Case}&\text{Weights}&\text{Upper bound for }\xi\\\hline
 AA&(1/2,1/2)&1-\gamma_3+z_3\le1+o(1)\\
 A_iB_j&(1/3,2/3)&1/2+\gamma_i/3+2b_1/3-k/24\\
 A_iC_j&(1/3,2/3)&1/2+\gamma_i/3+7b_1/12+k/12\\
 BB&(1/2,1/2)&1/2+7b_1/8+k/16\\
 B_iC_j&(1/2,1/2)&1/2+13b_1/16+5k/32\\
 CC&(1/2,1/2)&1/2+3b_1/4+k/4
\end{array}
\]
The mixed rows include both orders of $i,j=1,2$. For example, the
AB interpolation first gives
\begin{align*}
 \xi\le{}&\upsilon/3+2z_3/3+13k/24+b_1/12\\*
 &+(2\gamma_i+\gamma_j)/3.
\end{align*}
Substituting the bounds for $\upsilon$ and $z_3$ on the same subset
gives the AB row. For AC, replace $13k/24+b_1/12$ by $2k/3$.
The other rows follow directly from the bounds for $z_1+z_2$ and
$z_3$.
Since $\gamma_i\le5/12$ and $k\le b_1\le1/2+\epsilon$, the six rows
are bounded, respectively, by
\begin{gather*}
 1,\qquad35/36+2\epsilon/3,\qquad35/36+2\epsilon/3,\\
 31/32+15\epsilon/16,\qquad63/64+31\epsilon/32,
 \qquad1+\epsilon,
\end{gather*}
up to arbitrarily small losses. With $\epsilon=2\cdot10^{-6}$, all
except AA and CC are strictly below one. These bounds use only the
factor moments already stated, with no additional $L^1$ large-value
estimate.

Finally, suppose that the number of factors after merging is
$h\le4$. Write their length exponents as $a_j$. With $x_+=\max(x,0)$ denoting the positive part of a real number, define
\[
 \delta_j=\begin{cases}
 0,&\text{if factor $j$ is an original long factor},\\
 (2a_j-b_1)_+/4,&\text{otherwise}.
 \end{cases}
\]
Outside the pole part, \eqref{eq:069} controls the original long factors, and the ordinary fourth moment controls the others.
Thus every sampled fourth norm is at most
$N^{b_1/4+a_j/2+\delta_j+o(1)}$. All remaining factors have
$a_j\le b_0/2+o(1)$, so every $j$ satisfies
$\delta_j\le(b_0-b_1)/4+o(1)$.
H\"older's inequality over at most $N^{b_1+o(1)}$ character--time
points gives
\begin{align*}
 KA&\le N^{b_1(4-h)/4}
 \prod_{j=1}^hN^{b_1/4+a_j/2+\delta_j+o(1)}\\
 &\le N^{1/2+b_1+h(b_0-b_1)/4+o(1)}\\
 &\le N^{b_0+1/2+o(1)}\le N^{1+\epsilon+o(1)}.
\end{align*}
Since $e\le3k$, this also bounds $\xi\le\upsilon+k$.
The pole part is already bounded by~\eqref{eq:070}.

We now restore the auxiliary losses. Let $\kappa>0$ be the cutoff for small
factor exponents. There are at most ten factors. Replacing an ideal moment
bound by the trivial bound for the omitted small factors increases the value
exponent by at most $10\kappa$. For each remaining small factor the integer
$j$ chosen above is bounded in terms of $\kappa$, and is therefore fixed
independently of $N$.

Dividing a selected polynomial by $N^\rho$ replaces its value exponent $z$
by $z-\rho+o(1)$. The three exponents in Proposition~\ref{prop:CGL} consequently increase by at most
$4\rho$. Choose the free power loss in Proposition~\ref{prop:CGL} to contribute
at most $\varepsilon_C$ in the $N$-exponent. This is possible since
$r\mathcal T\le N^{1/2+2\eta+o(1)}$.
Thus a bound from Proposition~\ref{prop:CGL} for $e$ has an additional error at most
$4\rho+\varepsilon_C+o(1)$, and the corresponding bound for $\xi$ has error
at most $2\rho+\varepsilon_C/2+o(1)$.
All interpolations above are convex combinations, so they do not amplify
this error. The bound obtained when the extra condition fails needs no coefficient normalization.
The remaining mean-value and bounded-piece errors are $o(1)$ once these
parameters are fixed. Hence it is enough to choose
\[
 10\kappa+2\rho+\varepsilon_C/2<\varepsilon_*.
\]
This proves the stated estimates for the actual subsets.
\end{proof}

\subsection{The small-time kernel and the fourth-moment bound}
We now assemble the preceding estimates for the prime sum. The factor
partition and the long-factor moment input come from~\cite[Section 3]{MPR};
the character-energy estimate is Proposition~\ref{prop:CGL}, and the
fourth-moment relation is based on~\cite[Lemma 9.2]{CGL}. The proposition
below is the resulting estimate for our specified integration range.
It is not obtained merely by interpolating the pointwise prime estimate
with Parseval: that argument gives the weaker exponent
$1+2(19/24)=31/12$.

\begin{proposition}
\label{prop:high4}
Let $\mathfrak m_{\rm hi}$ consist of the minor arcs whose chosen
rational approximation satisfies $B_N>N^{0.46}$. Then
\[
 \int_{\mathfrak m_{\rm hi}}|V_w(\alpha)|^4\dd\alpha
 \le N^{2.5401+o(1)}.
\]
\end{proposition}

\begin{proof}
For large times, divide the expression in~\eqref{eq:066} by
$(r/q)KA^2$, the summand in~\eqref{eq:067}. The quotient is
\[
 \frac{r/q}{qTN}\frac{A^2\mathcal E(\mathcal W)}{K}
 =N^{2\xi_{\rm eff}-1-b_0+o(1)}.
\]
Summing~\eqref{eq:067} therefore gives the fourth-integral exponent
$1-b_0+2\xi_{\rm eff}$. By Lemma~\ref{an:budget}, this is at most
$3-b_0+2\varepsilon_*+o(1)\le2.54+2\varepsilon_*+o(1)$ except in the two-group B case. Before this additional $2\varepsilon_*$ loss, the latter is at most
\[
 3-\gamma-b_0/8\le3-5/12-0.46/8<2.54.
\]
The pole part is bounded by \eqref{eq:070}.

If $1\le T<N^\eta$, take $U=N^{2\eta}$ as in \eqref{an:eta}.
On $|t|>U$, repeated integration by parts gives an absolute time
integral $O_{\tau,w,L}(T^LU^{1-L})$. Choose a fixed sufficiently large
$L$ so that multiplication by $N^{5/4+\nu}$ leaves at most
$N^{-100}$. Since $U/T\ge N^\eta$, the choice is independent of $N$.
For the retained kernel, set $t=j+\delta$ and $z=\log y$. When the original scale is $T=1$, the variable $v$ ranges over $[0,1]$ instead of $[1/2,1]$; the following bound is uniform on either interval.
The character--time sum on a subset is
\begin{align*}
 Z(a_0,v)=\sum_{\mathcal W}&c_{\chi,j}\chi(a_0)N^{i(j+\delta)}\\
 &\times\int w(y)e(\pm Tvy)y^{i(j+\delta)}\frac{dy}{y}.
\end{align*}
Apply H\"older's inequality on the fixed compact $z$-interval and
absorb the unit-modulus factor $N^{i(j+\delta)}$ into the coefficients.
Then
\begin{align*}
 |Z(a_0,v)|^4&\ll\int g(z)
 \left|\sum_{\mathcal W}c'_{\chi,j}\chi(a_0)e^{ijz}\right|^4dz,\\
 |c'_{\chi,j}|&\le A,
\end{align*}
where $g$ is fixed, smooth, and nonnegative. Both $e(\pm Tv e^z)$
and the common factor $e^{i\delta z}$ have modulus one.
Character orthogonality, the rapid decay of the Fourier transform $\widehat g$ of $g$, and the
integer-translation correlation bound again give
$O(\varphi(q)A^4\mathcal E(\mathcal W))$.
The contribution of a small-time subset to the fourth integral is
therefore
\[
 N^{o(1)}\frac{r^2T}{q^3N}A^4\mathcal E(\mathcal W).
\]
Relative to the large-time formula with $T$ replaced by $U$, this
has precisely the extra factor $TU\le N^{3\eta}$.
Also $Q^2U\le N^{1+2\eta}$, so the large sieve square sum is at most
$N^{2+2\eta+o(1)}$ in place of $N^{2+o(1)}$.

The Heath--Brown factors have been merged using $QU$. Hence
$b_0\le1/2+\epsilon$, with $\epsilon=2\eta$, and the enlarged range
of Lemma~\ref{an:budget} applies to this decomposition.
Outside case B, the fourth-integral exponent is at most
$2.54+2\epsilon+2\varepsilon_*+o(1)$. In the two-group B case, the $\epsilon/2$ term
in the second-group mean value bound gives
\[
 3-\gamma-b_0/8+\epsilon
 \le3-5/12-0.46/8+\epsilon<2.54.
\]
The total fixed small-time loss is consequently at most
$2\epsilon+3\eta+2\eta=9\eta$.

The time-partition loss is also explicit. Since $r\mathcal T\le N^{1/2+2\eta+o(1)}$,
there are at most $N^{\zeta+o(1)}$ residue classes in each range. The inequality
$|\sum_{c=1}^L Z_c|^4\le L^3\sum_{c=1}^L|Z_c|^4$ costs at most
$N^{3\zeta+o(1)}$. The subsets across residue classes are disjoint, so their weighted
square sum is still bounded by the same large sieve; there is no further
factor equal to the number of residue classes.

Take
\[
 \eta=10^{-6},\qquad\varepsilon_*=10^{-6},\qquad\zeta=10^{-6}.
\]
Choose $\kappa,\rho,\varepsilon_C$ as in Lemma~\ref{an:budget}, and then choose
the stationary-phase and integration-by-parts orders after fixing $\tau,w$.
The total displayed fixed loss in the fourth-integral exponent is at most
\[
 9\eta+2\varepsilon_*+3\zeta=14\cdot10^{-6}<10^{-4}.
\]
The two-group B cases retain a fixed gap larger than these losses.
All remaining factors are powers of logarithms or $N^{o(1)}$ for the now fixed
parameters. They are absorbed by the difference between $2.540014$ and
$2.5401$ for sufficiently large $N$. This proves the claimed bound uniformly
in all ranges. We retain the more conservative exponent $2.5401$ throughout
the representation argument.
The noncoprime terms and the passage to $S_w$ are treated below.
\end{proof}

For fixed $q$, the terms of $V_w$ not coprime to $q$ are prime powers
whose underlying primes divide $q$. The sum of their absolute
coefficients is $O((\log N)^2)$. Each range contains $O(Q^2)$ fractions
and intervals of width $O(T/N)$, with $Q^2T\ll N$.
The total fourth integral of this error over all ranges is therefore
$N^{o(1)}$, in particular at most $N^{1/2+o(1)}$.
Put $H=V_w-S_w$. Parseval's identity and the number of higher prime
powers give
\begin{gather*}
 \int_{\mathbb T}|H|^2\le N^{1/2+o(1)},\qquad
 \|H\|_\infty\le N^{1/2+o(1)},\\
 \int_{\mathbb T}|H|^4\le N^{3/2+o(1)}.
\end{gather*}
Minkowski's inequality transfers Proposition~\ref{prop:high4} to
$S_w$. The prime-power error in the quadratic integral is also $o(N)$.
For the corresponding transfer in the weighted major-arc integral,
the preceding bounds $\|S_wb\|_2=O(\sqrt N)$ and $|W_{4,\nu}|\le|b|^2$,
together with $\|Hb\|_2\le\|H\|_2\le N^{1/4+o(1)}$, show by
Cauchy's inequality that the cross term is $o(N)$.

\subsection{Prime prefixes in the lower range}
\begin{lemma}[Kumchev {\cite[Theorem 2]{Kum}}]\label{lem:Kum}
Let $\alpha=a_0/q+\beta$ be reduced, with $q\le\sqrt x$ and
$|\beta|\le(q\sqrt x)^{-1}$. Then
\begin{equation}\label{eq:075}
 \left|\sum_{x<p\le2x}e(p\alpha)\right|
 \le x^{o(1)}\left(\frac{x}{\sqrt{B_x}}+
 \sqrt{B_x}x^{11/20}\right),\qquad B_x=\max(q,qx|\beta|).
\end{equation}
\end{lemma}
This is the $k=1$ specialization of the cited theorem, with its parameter
$Q_K=2B_x$. For $x\ge4$, one has $Q_K\le x$, $q\le Q_K$ and
$|q\alpha-a_0|<Q_K/x$. Its denominator $q+x|q\alpha-a_0|$ is comparable
to $B_x$. Changing $[x,2x)$ to $(x,2x]$ changes at most two unweighted
prime terms. The theorem has not yet inserted logarithmic or smooth weights.

\begin{proof}[Proof of~\eqref{eq:lowprime}]
Assume that $P_+<B_N\le N^{0.46}$.
To pass from complete intervals to smooth weights, we first bound
all prime prefixes. For each $Z\in[\tau N,N]$, decompose $(0,Z]$ into
the complete intervals $(Z/2^{j+1},Z/2^j]$ and write
$X=Z/2^{j+1}$ for the scale. Endpoint changes contribute a total
of $O(\log N)$.
If $X\ge q^2$, the original approximation applies and
\[
 (X/N)B_N\le B_X\le B_N.
\]
Geometric summation of the two terms over these scales gives
\[
 N^{o(1)}\left(\frac{N}{\sqrt{B_N}}
 +\sqrt{B_N}N^{11/20}\right).
\]
If $X<q^2$, choose a new reduced Dirichlet approximation
$\alpha=a'/q'+\beta'$, with $q'\le\sqrt X<q$.
The two reduced fractions are distinct, whence
\begin{gather*}
 1\le qq'|\beta-\beta'|
 \le B_NB'_X(1/N+1/X),\\
 B'_X=\max(q',q'X|\beta'|)\ge X/(2B_N).
\end{gather*}
Also $B'_X\le\sqrt X$, and \eqref{eq:075} applies to every
such $B'_X$, with no further lower-range restriction. Thus
\[
 \frac{X}{\sqrt{B'_X}}\le\sqrt{2XB_N},\qquad
 \sqrt{B'_X}X^{11/20}\le X^{4/5}.
\]
Since $X<q^2\le B_N^2\le N^{0.92}$, these terms are at most
$N^{0.69+o(1)}$ and $N^{0.736+o(1)}$, respectively, and summation
over the geometric scales preserves those bounds. Bounded scales
are negligible.
For the first family of scales, $B_N>N^{0.439}$ and
$B_N\le N^{0.46}$ give
\[
 N/\sqrt{B_N}\le N^{0.7805},\qquad
 \sqrt{B_N}N^{11/20}\le N^{0.78}.
\]
Consequently, the unweighted prefix
$A(Z)=\sum_{p\le Z}e(p\alpha)$ satisfies
\[
 \sup_{\tau N\le Z\le N}|A(Z)|\le N^{0.7805+o(1)}.
\]
Set $g(t)=(\log t)w(t/N)$. It vanishes at the endpoints of its
support and satisfies
$\int_{\tau N}^N|g'(t)|dt=O_{\tau,w}(\log N)$.
Stieltjes integration by parts yields
\[
 S_w(\alpha)=-\int_{\tau N}^NA(t)g'(t)dt,
 \qquad |S_w(\alpha)|\le N^{0.7805+o(1)}.
\]
This proves~\eqref{eq:lowprime}, and the prime-power transfer after Proposition~\ref{prop:high4} proves~\eqref{eq:fourth}.
The argument does not use the parity of $N$ or a recurrence weight.
\end{proof}

\subsection{A second moment on an arbitrary small set}
\begin{proof}[Proof of Lemma~\ref{lem:badset}]
In the lower range,~\eqref{eq:lowprime} gives
\[
 \int_{E_N\cap\m_{\rm lo}}|S_w|^2\dd\alpha
 \le\meas(E_N)\sup_{\m_{\rm lo}}|S_w|^2
 \ll N^{1.561-\delta+o(1)}.
\]
In the upper range, Cauchy's inequality and~\eqref{eq:fourth} give
\[
 \int_{E_N\cap\m_{\rm hi}}|S_w|^2\dd\alpha
 \le\meas(E_N)^{1/2}
 \left(\int_{\m_{\rm hi}}|S_w|^4\dd\alpha\right)^{1/2}
 \ll N^{(2.5401-\delta)/2+o(1)}.
\]
Adding the bounds proves~\eqref{eq:Etrans}. Both exponents are less than one when $\delta>0.561$.
The weaker condition $\delta>0.5401$ would control only the upper range.
The conclusion concerns exceptional minor arcs, not the whole set $E_N$, which may contain major arcs.
\end{proof}

\section{Conclusion}\label{sec:end}
We have proved that every sufficiently large positive integer is a sum
of two primes, $u$ Fibonacci numbers and $5-u$ Lucas numbers for each
$u\in\{1,2,3,4\}$. 

For $u=1,2,3$, pointwise bounds for the prime sum are enough to control
the exceptional minor arcs. For $u=4$, we also use a fourth-moment
bound. In every case, the major-arc contribution exceeds the possible
loss from the minor arcs.

The required numerical bounds are checked by exact integer and rational
calculations. The infinite tails and limiting arguments are proved in
the text. We fix all auxiliary parameters before taking $N$ sufficiently
large. Since there are only finitely many cases, one common threshold
for $N$ suffices.

We do not give an explicit value of this threshold. Our argument does
not cover five Fibonacci numbers alone, five Lucas numbers alone, or
a total of four recurrence terms. These cases require further estimates.

\appendix
\section{A cited mean and elementary recurrence estimates}\label{app:fib}
\subsection{The odd-Fibonacci mean}
\begin{lemma}[Sun {\cite[Proposition 4.1]{Sun}}]\label{lem:meanB}
With $b$ and $\mathfrak A_{b,J}$ as defined in the text,
\[
 \limsup_{3\mid J,\,J\to\infty}\mathfrak A_{b,J}<2.1985.
\]
\end{lemma}
Sun's class consists of $1\le j\le J$, $j\not\equiv2\pmod3$, with
values $F_{j+1}$ and normalization $2J/3$ when $3\mid J$. It is therefore
exactly the class used here.
The proof of the cited result includes finite computation. We retain that
external dependency, but do not repeat its tables or its enumeration in our
supplement. The slightly stronger rational output of an earlier duplicate
calculation is not needed anywhere in the representation argument.

\subsection{Prime periods}
The period $\pi(r)$ and rank $z(p)$ retain the definitions of
Section~\ref{sec:analytic}.
We record the standard period formulas and occurrence bounds from
Sun~\cite[Section 4.1 and Lemma 4.1]{Sun} and Wall~\cite{Wall}.
\begin{lemma}[Sun; Wall]\label{lem:period}
For an odd prime $p\ne5$, one has $z(p)\mid p-(5/p)$. The Legendre symbol
$(5/p)$ is $1$ or $-1$ according as $5$ is or is not a square modulo $p$.
The period is
\begin{equation}\label{eq:079}
 \pi(p)=\begin{cases}
 4z(p),&z(p)\text{ odd},\\
 z(p),&z(p)\text{ even and }F_{z(p)+1}\equiv1\pmod p,\\
 2z(p),&z(p)\text{ even and }F_{z(p)+1}\equiv-1\pmod p.
 \end{cases}
\end{equation}
Also $z(5)=5$ and $\pi(5)=20$. For odd squarefree $r$,
\begin{equation}\label{eq:080}
 \pi(r)=\lcm_{p\mid r}\pi(p).
\end{equation}
Each Fibonacci residue occurs at most four times in a prime period.
The same holds for Lucas residues when $p\ne5$.
\end{lemma}
Formula~\eqref{eq:080} is the Chinese remainder theorem applied to the two
coordinates of the recurrence state; the prime formulas are the cited
results. At $p=5$, one period of the Fibonacci sequence is
$0,1,1,2,3,0,3,3,1,4,0,4,4,3,2,0,2,2,4,1$.

\subsection{Euler products and elementary constants}
We use the explicit estimates of Rosser--Schoenfeld~\cite{RS}:
\[
\begin{aligned}
 x/\log x&<\pi_{\rm prime}(x) &&(x\ge17),\\
 \pi_{\rm prime}(x)&<1.25506x/\log x &&(x>1),\\
 \prod_{p\le x}(1-p^{-1})^{-1}
 &<e^\gamma\log x\left(1+\frac1{2\log^2x}\right)&&(x\ge286).
\end{aligned}
\]
Here $\pi_{\rm prime}$ counts primes and $\gamma$ is Euler's constant.
These are the only explicit prime-counting estimates used in the tail below.
They give
\[
 \pi_{\rm prime}(2P)-\pi_{\rm prime}(P)
 <1.51012P/\log P\qquad(P\ge10^8).
\]
The elementary scalar checks are included here so that they need not be
treated as unexplained software outputs.

Let $C^{(10000)}$ be the product defining $C_2$ restricted to $p\le10000$.
The finite product is checked by exact rational multiplication. Its omitted
factors have parameters with sum at most
$\sum_{k\ge5000}(4k^2)^{-1}\le1/19996$.
Since $\prod(1-x_i)\ge1-\sum x_i$ for $0\le x_i\le1$, this gives
\[
 0.66013<C^{(10000)}(1-1/19996)\le C_2<C^{(10000)}<0.66017.
\]
Only the finite product in this display is part of the computational checks.

The positive series for the logarithm, used in the proof of
Lemma~\ref{lem:fourdat}, gives by four terms for $\log2$ and two terms for
$\log(5/4)$
\[
 \log10=3\log2+\log(5/4)>
 \frac{176248}{76545}>2.30254,
 \qquad \log2<\frac{1910051}{2755620}<0.693148.
\]
For $H_n=\sum_{j=1}^n1/j$, termwise integration yields
\[
 H_n-\log n-\gamma
 =\sum_{k=n}^\infty\left(\log(1+1/k)-\frac1{k+1}\right)
 >\frac1{2(n+1)}.
\]
Taking $n=10$ gives
\[
 \gamma<\frac{3913409}{6735960}<0.581.
\]
For $x=0.581$, the exponential series and its geometric remainder give
\[
 e^x\le\sum_{j=0}^4\frac{x^j}{j!}
       +\frac{x^5}{120(1-x/6)}
 =\frac{387529231401098967}{216760000000000000}<1.788<1.79.
\]
In particular, $\log(10^8)>18.4$ and
$\log2/\log(10^8)<0.03763$.

Put $W_*(x)=\prod_{3\le p<x}(1+1/(p-2))$. The Euler-product identity
\[
 W_*(x)=\frac12\prod_{p<x}(1-p^{-1})^{-1}
               \prod_{3\le p<x}\left(1+\frac1{p(p-2)}\right)
\]
and the preceding estimates imply
\begin{equation}\label{eq:Wstar}
 W_*(x)<1.36\log x\qquad(x\ge10^8).
\end{equation}
For example, the required scalar comparison
$1.79(1+18.4^{-2})/(2\cdot0.66013)<1.36$ is the positive integer difference
\[
 34(66013)(8464)-25(89500)(8489)=2819588>0.
\]
The same argument, splitting the prime factors of a positive integer $H$
at $y=\log(3H)$, gives
\[
 \sum_{\substack{d\mid H\\d\text{ odd squarefree}}}\omega(d)
 \ll(\log\log(3H))^2.
\]
Indeed, the primes at most $y$ give $O(\log y)$; there are at most
$\log H/\log y$ larger prime factors, whose product of correction factors
is $\exp(O(1/\log y))$. The square on the displayed logarithm is a harmless
weakening. The case $H=1$ is immediate. This is the endpoint bound used in
Lemma~\ref{lem:exc} and Appendix~\ref{app:lucas}.

Finally, the four-term bound already proved gives
$\log\phi<0.481212<0.481213$ and
\[
 \frac{\log\phi}{\theta}
 <\frac{120303}{249875}<0.481453.
\]
The square-root comparisons in the Lucas tail reduce to
\[
\begin{aligned}
 (0.8525)^2-(1.51012)(0.481213)
 &=\frac{1671861}{25000000000}>0,\\
 (1.414214)^2-2&=\frac{309449}{250000000000}>0.
\end{aligned}
\]
Also $10^8/(10^8-2)<1.00000003$ and $2+\sqrt2<3.414214<3.42$.
The following multiplication, with every replacement upwards, proves the
final tail constant:
\[
\begin{aligned}
 &(8.16)(0.8525)(1.03763)(1.00000003)(0.0001)(3.414214)\\
 &\quad<\tfrac12(14.442)(3.42)(0.0001)
 =0.002469582<0.00247.
\end{aligned}
\]
To check the intermediate factor, use
\[
 (16.32)(0.8525)=13.9128,\qquad
 13.9128(1.038)(1.00000003)<14.442.
\]
No numerical logarithm or square-root routine is needed in these comparisons.

\section{The singular-series mean of odd Lucas differences}\label{app:lucas}
We prove the second estimate in~\eqref{eq:coll}. This supplies the constant $E_D$ used for the first two compositions.

\subsection{A reciprocal-period estimate}
Let $Q=\left(\begin{smallmatrix}1&1\\1&0\end{smallmatrix}\right)$, and suppose that $p>2$ and $\pi(p)=m$.
Then $m$ is even, since $\det Q=-1$ modulo $p$. All entries of $Q^m-I$ are divisible by $p$, whence
\[
 p^2\mid\det(Q^m-I)=2-L_m.
\]
For positive even $m$, one has $0<L_m-2<\phi^m$.
Consequently, for $P\ge10^8$,
\[
 \#\{P<p\le2P:\pi(p)=m\}\le\frac{m\log\phi}{2\log P}.
\]
Only even $m$ occur, so
\[
 \sum_{\substack{P<p\le2P\\\pi(p)\le Z}}\frac1{\pi(p)}
 \le\frac{Z\log\phi}{4\log P}.
\]
The explicit prime-counting inequalities in Appendix~\ref{app:fib} bound the remaining contribution by $1.51012P/(Z\log P)$.
Taking $Z=2\sqrt{1.51012P/\log\phi}$ gives
\begin{equation}\label{eq:Lperiod}
 \sum_{P<p\le2P}\frac1{\pi(p)}
 \le\frac{\sqrt{1.51012P\log\phi}}{\log P}.
\end{equation}
Both the evenness of the period and the square divisibility have been used.

\subsection{Periodic probabilities and the limiting average}
For odd squarefree $r$, let $T_r=\lcm(\pi(r),3)=\pi(2r)$. In a block of length $T_r$, let $n_v$ count the indices $j$ with $3\nmid j$ and $L_j\equiv v\pmod r$. Define
\[
 \delta_D(r)=\frac{\sum_{v\bmod r}n_v^2}{(2T_r/3)^2}.
\]
For fixed $r$, the normalized collision count tends to $\delta_D(r)$. Deleting the diagonal changes it by $O(J^{-1})$.

For odd $p\ne5$, a prescribed Lucas value occurs at most four times in a complete Fibonacci state period, as shown in the proof of Lemma~\ref{lem:exc}. In any interval of $J$ indices, its total number of occurrences is at most $4J/\pi(p)+4$. Restricting to $3\nmid j$ can only decrease this count. Division by the $2J/3$ selected indices gives the finite-window bound
\[
 \delta_{D,J}(p)\le\frac6{\pi(p)}+\frac6J.
\]
For $p\mid r$, a collision modulo $r$ implies one modulo $p$. We use this estimate only for tail primes larger than a fixed cutoff, so the prime $5$ causes no problem.

Let $H_J$ be the product of the nonzero differences $|L_i-L_j|$, with $i,j$ among the selected indices. Then $\log H_J=O(J^3)$ and
\[
 \sum_{\substack{r\mid H_J\\r\text{ odd squarefree}}}\omega(r)\ll(\log J)^2,
\]
by the Euler-product argument in Appendix~\ref{app:fib}. Thus all finite-window endpoint errors together contribute $O((\log J)^2/J)=o(1)$.
Group the remaining terms by their largest prime factor. Equation~\eqref{eq:Lperiod} and the bound $W_*(x)<1.36\log x$ give a summable tail independent of $J$.
For each fixed prime cutoff, only finitely many squarefree moduli remain, so their limits can be taken term by term. Letting the cutoff grow proves the upper bound by the complete series. Conversely, retaining any finite collection of positive terms gives the corresponding lower limit. Hence the limit exists and
\begin{equation*}
 \lim_J\mathfrak A_{d,J}
 =2C_2\sum_{\substack{r\ge1\\r\text{ odd squarefree}}}\omega(r)\delta_D(r).
\end{equation*}
No uncontrolled exchange of an infinite sum and a limit is used.

\subsection{Finite terms and all omitted moduli}
Take $B_0=10^5$. The complete periodic histograms give every term with $r\le B_0$ exactly.
If $r>B_0$ and its largest prime factor $p$ is at most $B_0$, then $p\ge7$ and there is a unique factorization
\[
 r=gpe,\qquad g\in\{1,3,5,15\},\quad(e,15)=1,\quad P^+(e)<p.
\]
Here $P^+(1)=1$. Retaining the joint congruence condition gives $\delta_D(r)\le\delta_D(gp)$; the latter is not replaced by a product of independent probabilities.
Define
\[
 W_{15}(p)=\prod_{7\le\ell<p}\left(1+\frac1{\ell-2}\right),\qquad
 W_{15,p}(Y)=\sum_{\substack{e\le Y,\ e\text{ odd squarefree}\\(e,15)=1,\ P^+(e)<p}}\omega(e),
\]
where the product is over primes. The total of these omitted terms is at most
\[
 \sum_{7\le p\le B_0}\sum_{g\mid15}
 \frac{\omega(g)\delta_D(gp)}{p-2}
 \left(W_{15}(p)-W_{15,p}\!\left(\frac{B_0}{gp}\right)\right).
\]
This finite expression bounds all moduli in that class, not just a larger sample of moduli.
For $B_0<p\le10^8$, use $6W_*(p)/((p-2)\pi(p))$ and the exactly computed prime periods.

Finally, applying~\eqref{eq:Lperiod} on $P=2^k10^8$ gives
\[
 \sum_{p>10^8}\frac{6W_*(p)}{(p-2)\pi(p)}<0.00247.
\]
An explicit upper bound is
\[
 \frac{8.16\sqrt{1.51012\log\phi}\,(1+\log2/\log X_0)}
 {(1-2/X_0)(1-1/\sqrt2)\sqrt{X_0}},\qquad X_0=10^8.
\]
The positive-term rational series used in Appendix~\ref{app:fib} verify that this is less than $0.00247$.

The four parts, before multiplication by $2C_2$, have the following rational upper bounds. The displayed finite decimals are rounded upwards.
\begin{center}\small
\begin{tabular}{p{.69\linewidth}r}
\toprule
Part & Upper bound\\\midrule
Complete positive sum for $r\le10^5$ & 1.691933204677\\
Omitted $r>10^5$ with $P^+(r)\le10^5$ & 0.012881438111\\
Prime-period sum for $10^5<p\le10^8$ & 0.000371798387\\
Analytic tail for $p>10^8$ & 0.002470000000\\\bottomrule
\end{tabular}
\end{center}
The exact directed calculations use scale $10^{24}$. Using the saved integer upper bounds before decimal display, and $2C_2<1.32034$, gives
\[
 \lim_J\mathfrak A_{d,J}<2.254687105540<2.255.
\]
Even the sum of the displayed upward-rounded entries, multiplied by $1.32034$, is strictly below $2.255$. This proves the required estimate without relying on the last displayed digits.

The regeneration code computes the prime periods by fast doubling. It checks that the candidate power of the recurrence matrix is the identity and that division by any prime factor destroys this property. For all primes at most $10^5$, full recurrence iteration gives an independent cross-check. The finite computations do not replace the analytic tail proof above.

\section{Finite certificates and reproducibility}\label{app:cert}
A certificate here means a fixed list of integers and rational identities whose exact verification proves a stated quantitative bound. The supplement contains only these lists and the source code needed
for the new finite inequalities. All $34$ moment vectors are on the $377$
representatives of $\mathcal S_+$; the original full-box vectors are unnecessary.
The Lucas output tables can be regenerated from the supplied source files.

\subsection{Parameters for the common large-deviation estimate}
The following are the exact rational parameters in~\eqref{eq:cert}.
Coordinates are ordered first by $m=0$, $n=0,\ldots,18$, then by $m=1,\ldots,18$ with $n$ increasing along the corresponding row of $\mathcal S_0$.
\begingroup\small
\setlength{\tabcolsep}{4pt}
\begin{longtable}{rrrrrrr}
\caption{Parameters for the $32$ common matrix certificates.}\\
\toprule
$i$ & $t_a$ & $t_b$ & $t_c$ & $t_d$ & $\rho_i$ & $H_i$\\\midrule
\endfirsthead
\toprule
$i$ & $t_a$ & $t_b$ & $t_c$ & $t_d$ & $\rho_i$ & $H_i$\\\midrule
\endhead
0 & 0.8 & 0 & 0 & 0 & 1.1689052 & 0.3328785 \\
1 & 0 & 0.8 & 0 & 0 & 1.502687 & 0.4166077 \\
2 & 0 & 0 & 0.8 & 0 & 1.1689052 & 0.3328785 \\
3 & 0 & 0 & 0 & 0.8 & 1.502687 & 0.4166077 \\
4 & 0.16 & 0.16 & 0.4 & 0.4 & 1.3933263 & 0.3914207 \\
5 & 0 & 0.28 & 0.56 & 0.3 & 1.4408379 & 0.4025977 \\
6 & 0.56 & 0 & 0.58 & 0.3 & 1.4656941 & 0.4082991 \\
7 & 0 & 0.28 & 0 & 0.58 & 1.4455748 & 0.4036918 \\
8 & 0.55 & 0.28 & 0 & 0.31 & 1.4999362 & 0.415997 \\
9 & 0.6 & 0.6 & 0.3 & 0.3 & 3.5442808 & 0.7026349 \\
10 & 0.3 & 0.3 & 0.6 & 0.6 & 3.5233196 & 0.7006577 \\
11 & 0.4 & 0.4 & 0.4 & 0.4 & 2.634295 & 0.6037282 \\
12 & 0.139918 & 0.201013 & 0.0841 & 0.547221 & 1.4414929 & 0.4027492 \\
13 & 0.098433 & 0.22296 & 0.130803 & 0.52019 & 1.4339461 & 0.4009995 \\
14 & 0.277748 & 0.137175 & 0.302109 & 0.426796 & 1.4224535 & 0.3983172 \\
15 & 0.44882 & 0.053936 & 0.420928 & 0.36973 & 1.4351432 & 0.4012776 \\
16 & 0.421811 & 0.067831 & 0.445834 & 0.35548 & 1.4312239 & 0.4003661 \\
17 & 0.49689 & 0.029329 & 0.518389 & 0.324755 & 1.4441291 & 0.4033582 \\
18 & 0.370703 & 0.084563 & 0.066622 & 0.561945 & 1.4588146 & 0.4067308 \\
19 & 0.394138 & 0.085278 & 0.134873 & 0.510656 & 1.4457649 & 0.4037356 \\
20 & 0.281291 & 0.139381 & 0.285636 & 0.431017 & 1.4228998 & 0.3984217 \\
21 & 0.149539 & 0.19701 & 0.456302 & 0.35356 & 1.4204091 & 0.3978377 \\
22 & 0.123349 & 0.214102 & 0.43616 & 0.359368 & 1.420072 & 0.3977586 \\
23 & 0.039892 & 0.262831 & 0.505498 & 0.319976 & 1.4285595 & 0.399745 \\
24 & 0.171451 & 0.195385 & 0.39914 & 0.371603 & 1.4178224 & 0.3972301 \\
25 & 0.170002 & 0.193456 & 0.399903 & 0.374059 & 1.4180479 & 0.3972832 \\
26 & 0.170197 & 0.19438 & 0.403287 & 0.371292 & 1.4180138 & 0.3972751 \\
27 & 0.167042 & 0.196212 & 0.398668 & 0.373321 & 1.4179647 & 0.3972636 \\
28 & 0.166854 & 0.196286 & 0.401417 & 0.37198 & 1.4180046 & 0.397273 \\
29 & 0.168136 & 0.194783 & 0.401203 & 0.373 & 1.4180519 & 0.3972841 \\
30 & 0.169337 & 0.195089 & 0.40076 & 0.372248 & 1.4179629 & 0.3972632 \\
31 & 0.168339 & 0.195147 & 0.402023 & 0.372096 & 1.4180246 & 0.3972777 \\
\bottomrule
\end{longtable}
\endgroup
The parameters and all reduced vectors are in \path{vectors.json}. The two
files \path{partition_even.json} and \path{partition_odd.json} give the complete
amplitude partitions. The local rows and their six-block sums are in
\path{major_rows.csv} and \path{major_summary.json}. These paths are relative to
\path{anc/proof_data/}.

\subsection{Rational enclosures and acceptance tests}
For the parameters used here, the positive-term series give the enclosures
\begin{align*}
 E^+(t)&=\sum_{j=0}^{13}\frac{t^j}{j!}
 +\frac{t^{14}}{14!}\frac1{1-t/15},\\
 I_k^-(t)&=\sum_{j=0}^{5}\frac{(t/2)^{2j+k}}{j!(j+k)!},\\
 I_k^+(t)&=I_k^-(t)+\frac{(t/2)^{12+k}}{6!(6+k)!}
 \frac1{1-t^2/(4\cdot7\cdot(7+k))}\qquad(1\le k\le3).
\end{align*}
The tails are bounded by the first omitted term times a geometric series.
Hence $e^t\le E^+(t)$ and $I_k^-(t)\le I_k(t)\le I_k^+(t)$.
For the constant Fourier coefficient use
$c_0(t)\le E^+(t)-2\sum_{k=1}^3I_k^-(t)$.
At $t=0$, $c_0=1$ and all other coefficients are zero exactly.
Every coefficient is rounded upwards at scale $10^{12}$.
The six denominators in a period are then cleared, so each comparison with a positive vector is an integer inequality. Fourteen exponential terms and six Bessel terms suffice for every parameter in the table, whose entries are at most $0.8$.

For logarithms, put $z=(x-1)/(x+1)$ and use
\[
 \log x=2\sum_{j\ge0}\frac{z^{2j+1}}{2j+1},\qquad x>1,
\]
with a positive geometric upper bound for the tail. The rational comparisons
$\phi<1.618034$, $\pi<22/7$ and
$\cos(\pi/2000)\ge1-\tfrac12(22/14000)^2$
provide lower bounds for the rates in~\eqref{eq:rate}.
The checker verifies vector positivity, every matrix coordinate, the rate inequalities, complete box coverage, nonnegative polynomial coefficients and the coefficient-sum remainder bound. It then recomputes all required cyclic convolutions and checks the strict main-minus-minor margins.

The supplementary module \path{finite.py} also checks the fixed
power-loss budget in the fourth-moment proof and recomputes the two rates for
$u=4$ using only the rational bounds $\pi<22/7$ and $\phi<1.618034$.
These checks concern the stated scalar inequalities, not the analytic
validity of the fourth-moment estimate.

The positive vectors may be found by numerical optimization. Their acceptance does not use a floating-point eigenvalue or an optimality assertion. Likewise, the amplitude certificates prove inequalities on complete boxes, not merely at sampled points.

\subsection{What is checked, and what is proved on paper}\label{sec:trust}
The finite checks establish lower bounds for the major-arc constants and
upper bounds for the recurrence-difference means and matrix growth rates.
The latter rates determine the large-deviation exponents. These bounds are
needed in specific ranges: a weaker lower bound, a larger upper bound, or
an insufficient decay exponent could prevent the final comparison from being
positive. Thus the calculations are necessary quantitative steps, not tests
of whether finitely many integers have representations.

Sections~\ref{sec:mainpf} and~\ref{sec:proofs} prove the implication from the
verified inequalities to the representation theorem. Every accepted finite
comparison is exact, and all discarded infinite terms have explicit analytic
bounds. Computer assistance therefore changes how these finite steps are
checked, not their status as rigorous inequalities. The checking tasks are:
\begin{enumerate}
\item For each of $34$ fixed parameter rows, the checker constructs the
$377$-state transitions from~\eqref{eq:redact}, forms rational upper bounds
for the Bessel coefficients, and verifies every strict integer upper-vector
inequality. It also checks the stated large-deviation rate.
\item For each parity it reconstructs the complete bisection of $[0,1]^4$.
It checks all polynomial coefficients and the $\ell^1$ bound for the residual
in~\eqref{eq:poly}. Thus a continuum is covered by proved inequalities, not
by sampling finitely many amplitudes.
\item It generates every odd squarefree modulus $r\le300$, iterates the
recurrence until its initial state returns, and recomputes the histograms and
all needed cyclic convolutions. Parity and total mass are checked at each
stage. It verifies the downward-rounded rows~\eqref{eq:majorint} and their
sums, rather than trusting a displayed decimal.
\item In full mode it regenerates the odd-Lucas histograms, the joint
small-prime histograms, and the prime periods through $10^8$. It checks
minimal periods, the completeness of the moduli, all outward rounding, and
the bound $2.255$. The finite Euler product for $C_2$ is checked separately.
\end{enumerate}
All accept/reject decisions use integers or rational numbers. Floating-point
optimization is not invoked by the checker, and a floating-point eigenvalue
is never accepted as a matrix bound. The $34$ supplied vectors have $377$
coordinates each, all between $1$ and $10^{10}$. Each is an explicit positive vector proving the matrix bound: no reproduction of their original search is necessary.

The direction of the matrix action can be inspected without reconstructing a
large matrix. For source $(m,n)$ and destination $(m',n')$, the inverse of a
step with digits $r,s$ is
\[
 (m,n)=(n'+s,\ m'-n'-r-3s).
\]
The checker implements both this pullback and the source-to-destination rule
and compares their actions. Quotient multiplicities are justified by
Lemma~\ref{lem:reduce}, not guessed from a symmetry factor.
The recurrence-return test likewise has an elementary termination bound:
modulo $q$ the state update is invertible on the $q^2$ pairs, so a state
returns within $q^2$ steps. Fixed guards in the larger routines cause failure
if exceeded; they are not silently assumed period theorems.

The C++ enumeration uses fixed-width integers only within explicit bounds.
In the histogram routines the guards give a mass at most $9.6\cdot10^6$;
a sum of histogram squares is at most the square of that mass. These quantities
are below $2^{63}$. In the prime-period routine, the checked Euler-product
bound is at most $100\cdot10^{24}$. With $p\le10^8$, its largest product
before division is below $10^{34}<2^{128}$. The rational aggregation is
independently reconstructed with Python integers of arbitrary precision.
Small periods are cross-checked by direct iteration, separately from the
rank/fast-doubling calculation; selected collision probabilities are also
recomputed from the Lucas recurrence rather than the Fibonacci state formula.

File lengths and SHA256 hashes specify the supplied files and the expected
regenerated outputs. Hashes are integrity checks, not mathematical evidence
for a bound. Full verification recomputes the arithmetic as well as comparing
files. The checker fails on missing or malformed required data, and the
reproduction tests include rejection of a nonpositive vector entry and of a
missing partition box. No finite search for representations of integers
is used anywhere.

The remainder of the proof is not certified by these programs. The state
restriction, symmetry, parity compression, short scalar comparisons,
exceptional-conductor argument, analytic prime estimates, and infinite Lucas
tail are proved in the text. The odd-Fibonacci mean is explicitly cited from
Sun; removing its duplicate enumeration does not turn that cited input into
a new computer-free proof. These distinctions make the finite obligations
reproducible without claiming a formal verification of the whole paper.

\subsection{Compact supplementary files and verification}
The accompanying \path{anc/proof_data/} directory contains one reduced-vector
file, two amplitude-partition files, the short major-arc integer table, and
verification/generation source code. The four large Lucas output files are
not included: their complete generation requires no external data. Obsolete
certificates for other statements, historical reports, duplicate archives,
full-box vectors, and the repeated Fibonacci collision enumeration are omitted.
The parameters needed to reproduce every new finite input remain fixed in
the supplied files and in the paper.

From that directory run
\begin{Verbatim}[fontsize=\small]
python verify.py
\end{Verbatim}
This verifies the reduced vectors, the complete amplitude partitions, all
local convolutions and the scalar comparisons. It explicitly reports that
the large Lucas enumeration was not performed. To check all new finite inputs,
run
\begin{Verbatim}[fontsize=\small]
python verify.py --full --report /path/outside/proof_data/report.json
\end{Verbatim}
Python 3.10 or later and its standard library suffice for the local mode;
NumPy is not required. Full mode additionally requires a C++17 compiler with
unsigned $128$-bit integer support, selected by \texttt{CXX} or \texttt{g++}.
All generated tables are placed in a temporary directory. The supplied inputs
are read-only, checked again at completion, and never overwritten with a new
runtime log. Verification conditions remain active with Python optimization
as well. A successful full run establishes the new finite inequalities just
listed, not the analytic theorems or the cited Fibonacci mean.

\Needspace{11\baselineskip}
\section*{Acknowledgements}
 This work was motivated by Sun's paper \cite{Sun}. The authors
described a circle-method approach to the case of two primes, three
Fibonacci numbers and two Lucas numbers. ChatGPT (OpenAI) filled in the 
details of the argument along this approach. The cases with one
Fibonacci number and four Lucas numbers, and with two Fibonacci
numbers and three Lucas numbers, were developed with similar AI
assistance. Rethlas supplied the argument for the case with four
Fibonacci numbers and one Lucas number. ChatGPT also assisted with
the writing and revision of the manuscript.

The authors have checked the analytic arguments and the mathematical
basis of the verification algorithms for all four cases. The extensive
finite computations are carried out by the accompanying programs
using exact integer and rational arithmetic, with explicit bounds
for truncation and rounding errors. These computations have not been
repeated in full by hand. The authors take responsibility for the
mathematical content and conclusions of the paper.

\end{document}